\documentclass[a4paper,10pt]{article}
\usepackage[utf8]{inputenc}

\usepackage{ucs}
\usepackage{amsmath}
\usepackage{amsfonts}
\usepackage{amssymb}
\usepackage{amsthm}
\usepackage[english]{babel}
\usepackage{fontenc}
\usepackage{graphicx}
\usepackage{geometry}
\usepackage{hyperref}
\usepackage{fancyhdr}

\fancypagestyle{specialfooter}{%
  \fancyhf{}
   
  \fancyfoot[L]{\small\quad This work was supported by the UK Engineering and Physical Sciences Research Council (EPSRC) studentship award 1789662.}
}
\newtheorem{theorem}{Theorem}
\newtheorem*{theorem*}{Theorem}
\newtheorem{lemma}{Lemma}
\newtheorem*{lemma*}{Lemma}

\newtheorem{proposition}{Proposition}
\newtheorem*{proposition*}{Proposition}
\title{Inversion of the $j$--function and testing complex multiplication}
\author{John Armitage} 
\begin{document}

\maketitle
\thispagestyle{specialfooter}
\section{Introduction}
In this paper we develop an algorithm to invert the $j$--function in quasilinear time, and give an application to testing whether an elliptic curve has complex multiplication.

It is known that the inverse of $j$ is equal to a ratio of two Gaussian hypergeometric functions, though we do not believe any analysis has been made of the running time or required precision to invert $j$ by this method.
Our method is similar to that of \cite{fasttheta}, which makes use of addition formulae between Jacobi's theta functions in order to reduce the argument to a fixed compact set -- here we repeatedly make use of the modular polynomial 
\begin{align*}
 \Phi_2(X,Y)=X^3&+Y^3-X^2Y^2+1488X^2Y+1488XY^2-162000X^2-162000Y^2\\+&40773375XY+8748000000X+8748000000Y-157464000000000,
\end{align*}
which has the property that the roots of $\Phi_2(j(\tau),z)$ are $j(2\tau)$, $j\left(\frac{\tau}{2}\right)$, and $j\left(\frac{\tau+1}{2}\right)$,
 to either compute $j(2^k\tau)$, the logarithm of which is a close approximation to $-2^{k+1}\pi\tau$, or to manipulate the argument of $j$ into a compact set to which Newton's method may be applied.

 We define the \emph{regulated precision} of an approximation $\tilde{\alpha}$ to $\alpha$ to be
\begin{equation*}
 \frac{\lvert\alpha-\tilde{\alpha}\rvert}{\max\{1,\lvert\alpha\rvert\}},
\end{equation*}
and denote by $M(P)$ the computational complexity of multiplication of two $P$--bit integers, which by a recent result of Harvey and Hoeven \cite{fastmult} may be taken to be $O(P\log P)$. We obtain the following,
\begin{theorem}
 Suppose that $\tilde{j}$ is an approximation to $j(\tau)$, $\tau\in\mathcal{F}$, of regulated precision $2^{-P}$, with $P\geq400$. Let $Q=P/6$ if $\lvert \tau-i\rvert\leq2^{-30}$ or $\left\lvert \tau-\frac{\pm1+i\sqrt{3}}{2}\right\rvert\leq2^{-30}$, and $Q=P-\max\{11\log P,100\}$ otherwise. Then we may obtain an approximation to $\tau$ of relative precision $2^{-Q}$ in time
 \begin{equation*}
  O(M(P)\log (P)^2).
 \end{equation*}
\end{theorem}
The $j$--function has two ramification points in its fundamental domain, which entails the loss of precision in its inversion when $j$ is close to $0$ or $1728$.
We apply this algorithm to test for complex multiplication of elliptic curves -- given an approximation to the $j$--invariant of an elliptic curve and a bound upon its height and degree, we may invert it and determine if the inverse is a quadratic irrational, determining also the discriminant,
\begin{theorem}
 Suppose that $j(\tau)$ is the $j$--invariant of an elliptic curve $E$, with $j(\tau)$ of degree bounded by $d$ and height bounded by $H\geq e^e$. Then it may be determined from $d$, $H$, and an approximation to $j$ of regulated precision $2^{-300d^2\log H(\log d+\log\log H)^2-200}$ whether $E$ has complex multiplication, and if so the associated discriminant, in time, letting $T=d^2\log H(\log d+\log\log H)^2$, 
\begin{equation*}
 O(M(T)(\log T)^2).
\end{equation*}
\end{theorem}
Previous methods include that of \cite{achter}, based on reduction of elliptic curves at primes, which has an unconditional running time of $O(H^{cd})$, and assuming the Generalized Riemann Hypothesis a running time of $O(d^2(\log H)^2)$, and two tests of \cite{xavier}, comprising a deterministic algorithm based on Galois representations associated to torsion points, with running time $O(d^{c_1}(\log H)^{c_2})$, with an ineffective implicit constant, and a probabilistic algorithm, also of polynomial running time.

We note that one may also apply our algorithm for the inversion of $j$ to detecting isogenies between two elliptic curves of running time $O(N\log N\log\log N)$, where the degree of the isogeny is bounded by $N$, though our implicit constant is ineffective as explicit bounds on the coefficients of modular polynomials $\Phi_N(X,Y)$ for composite indices are not currently available.
\section{Preliminaries}
We will denote by $\mathcal{F}$ the usual fundamental domain of $j(z)$, $\left\{z|-\frac{1}{2}<\text{\normalfont Re}(z)\leq\frac{1}{2},\lvert z\rvert>1\right\}\cup\{z|\lvert z\rvert=1,\text{\normalfont Re}(z)\geq0\}$.
Throughout we will make use of the following results,
\begin{lemma}[Lemma 1 of \cite{j2079}]\label{j_q}
If $\tau\in\mathcal{F}$
 \begin{equation*}
  \left\lvert j(\tau)-e^{-2\pi i\tau}\right\rvert\leq2079.
 \end{equation*}
\end{lemma}

\begin{theorem}[Kantorovich, \cite{kantorovich}]\label{kantorovich}
 Let $F:S(x_0,R)\subset X\to Y$ have a continuous Fr\'{e}chet derivative in $\overline{S(x_0,r)}$. Moreover, let \emph{(i)} the linear operation $\Gamma_0=[F'(x_0)]^{-1}$ exist; \emph{(ii)} $\lVert\Gamma_0 F(x_0)\rVert\leq\eta$; \emph{(iii)} $\lVert\Gamma_0F''(x)\rVert\leq K$ $(x\in\overline{S(x_0,r)})$. Now, if
 \begin{equation*}
  h=K\eta\leq\frac{1}{2}
 \end{equation*}
and 
\begin{equation*}
 r\geq\frac{1-\sqrt{1-2h}}{h}\eta,
\end{equation*}
then $F(x)=0$ will have a solution $x^*$ to which the Newton method is convergent. Here,
\begin{equation*}
 \lVert x^*-x_0\rVert\leq r_0.
\end{equation*}
Furthermore, if for $h<\frac{1}{2}$,
\begin{equation*}
 r<r_1=\frac{1+\sqrt{1-2h}}{h}\eta,
\end{equation*}
or for $h=\frac{1}{2}$
\begin{equation*}
 r\leq r_1,
\end{equation*}
the solution $x^*$ will be unique in the sphere $\overline{S(x_0,r)}$. The speed of convergence is characterized by the inequality
\begin{equation*}
 \lVert x^*-x_k\rVert\leq\frac{1}{2^k}(2h)^{2^k}\frac{\eta}{h}
\end{equation*}
for $k=0,1,2,\ldots$.
\end{theorem}
We note that the condition 
\begin{equation*}
 r\geq\frac{1-\sqrt{1-2h}}{h}\eta
\end{equation*}
may be replaced with
\begin{equation*}
 r\geq 2\eta.
\end{equation*}

\section{Inversion of $j(z)$}
Firstly, if $\lvert j \rvert\leq 2^{-P/2}$ or $\lvert j-1728\rvert\leq 2^{-P/3}$, we return $\tau = \frac{1+i\sqrt{3}}{2}$ or $\tau = i$ respectively, and otherwise, we split the fundamental domain of $j$ into $4$ sections -- $\text{\normalfont Im}(\tau)\geq 3$, $\lvert\tau-i\rvert\leq 2^{-{31}}$, $\left\lvert\tau-\frac{1+\sqrt{3}}{2}\right\rvert\leq2^{-31}$, $\left\lvert\tau-\frac{-1+\sqrt{3}}{2}\right\rvert\leq2^{-31}$ and the remaining compact subset of the fundamental domain. We may determine $\tau$ with sufficient precision by taking a low precision inverse via the following expression for $j^{-1}$ in terms of Gaussian hypergeometric functions -- with $\alpha$ a solution to 
\begin{equation*}
 j(\tau)=\frac{1728}{4\alpha(1-\alpha)},
\end{equation*}
$\tau$ is equal to either
\begin{align*}
 i\frac{_2F_1\left(\frac{1}{6},\frac{5}{6},1,1-\alpha\right)}{_2F_1\left(\frac{1}{6},\frac{5}{6},1,\alpha\right)}.
\end{align*}
or the negative of its inverse.
If $j$ is sufficiently large, or close to $0$ or $1728$, then we do not need to evaluate this in order to determine which section of the fundamental domain $\tau$ lies in, so we need only compute the above to a fixed precision in a compact set, at points which it is easily shown are bounded away from zeros of $_2F_1\left(\frac{1}{6},\frac{5}{6},1,z\right)$, so this takes only constant time.
\subsection{Large $j$}

Throughout this section we assume $\text{\normalfont Im}(\tau)\geq 3$, and will repeatedly make use of the consequent fact that $\lvert j(\tau)\rvert\geq 10^8$. We will make use of the modular polynomial $\Phi_2(X,Y)$ to obtain an approximation to $j(2\tau)$, and repeat the process until we have an approximation to $j(2^k\tau)$, where $k$ is sufficiently large, at which point taking the logarithm of $j(2^k\tau)$ gives a good approximation to $-2^{k+1}\pi\tau$, owing to the $q$--series of $j$.

\begin{proposition}\label{phi2discrepancy}
 Let $j(\tau)=j$, and suppose that $\text{\normalfont Im}(\tau)\geq3$ and $\tilde{j}$ is an approximation to $j$ of relative precision at least $2^{-P}$, with $P\geq300$. Then the largest root, in absolute value, of $\Phi_2(\tilde{j},z)$ is an approximation to $j(2\tau)$ of relative precision at least $2^{-P+2}$.
\end{proposition}
\begin{proof}
 Let $f(z) = \Phi_2(j,z)$ and $g(z) =\Phi_2(\tilde{j},z)$. We first bound the coefficients of $f(z)-g(z)$. For the coefficient of $z^2$, we have, with $\tilde{j}=j+\delta$, 
 \begin{align*}
  \left\lvert -j^2+1488j-162000-(-(j+\delta)^2+1488(j+\delta)-162000)\right\rvert &= \lvert 2\delta j+\delta^2+1488\delta\rvert\\
  &\leq 2.1\lvert\delta\rvert\lvert j\rvert,
 \end{align*}
as $\lvert \delta\rvert\leq2^{-300}\lvert j\rvert$ and $\lvert j\rvert\geq 10^8$.
For the coefficient of $z$, we have
\begin{align*}
  &\left\lvert 1488j^2 +40773375j+8748000000-(1488(j+\delta)^2+40773375(j+\delta)+8748000000)\right\rvert\\
  &\quad\quad\quad= \lvert 2976\delta j+1488\delta^2+40773375\delta\rvert\\
  &\quad\quad\quad\leq 3210\lvert\delta\rvert\lvert j\rvert.
 \end{align*}
 For the constant coefficient, we have
 \begin{align*}
  &\left\lvert 3\delta j^2+3\delta^2j+\delta^3-16200\delta j-16200\delta^2+8748000000\delta\right\rvert\\
  &\quad\quad\quad\leq 3.4\lvert\delta\rvert\lvert j\rvert^2.
 \end{align*}
 We now bound the values of $g(j(\tau'))$ for $\tau'\in\left\{2\tau,\frac{\tau}{2},\frac{\tau+1}{2}\right\}$.
 For $j(2\tau)$, as $\lvert j(2\tau)\rvert\leq 1.02\lvert j\rvert^2$,
 \begin{align*}
  \lvert g(j(2\tau))\rvert &= \lvert g(j(2\tau))-f(j(2\tau))\rvert\\
  &\leq 2.1\lvert\delta\rvert\lvert j\rvert(1.02\lvert j\rvert^2)^2+3210\lvert\delta\rvert\lvert j\rvert\cdot1.02\lvert j\rvert^2+3.4\lvert\delta\rvert\lvert j\rvert^2\\
  &\leq 2.2\lvert\delta\rvert\lvert j\rvert^5.
 \end{align*}
 For $g(j(\tau'))$, $\tau'\in\left\{\frac{\tau}{2},\frac{\tau+1}{2}\right\}$,  firstly, the absolute values of $j\left(\frac{\tau}{2}\right)$ and $j\left(\frac{\tau+1}{2}\right)$ are bounded above and below by $e^{2\pi\text{\normalfont Im}(\tau)/2}\pm2079$, and the absolute value of $j$ is bounded above and below by $e^{2\pi\text{\normalfont Im}(\tau)}\pm2079$. As $\text{\normalfont Im}(\tau)\geq 3$, these yield 
\begin{equation}
0.83\lvert j\rvert^{1/2}\leq\left\lvert j\left(\frac{\tau+1}{2}\right)\right\rvert, \left\lvert j\left(\frac{\tau}{2}\right)\right\rvert\leq 1.17\lvert j\rvert^{1/2},
\end{equation}
so 
\begin{align*}
  \lvert g(j(\tau'))\rvert &= \lvert g(j(\tau'))-f(j(\tau'))\rvert\\
  &\leq 2.1\lvert\delta\rvert\lvert j\rvert(1.17\lvert j\rvert^{1/2})^2+3210\lvert\delta\rvert\lvert j\rvert(1.17\lvert j\rvert^{1/2})+3.4\lvert\delta\rvert\lvert j\rvert^2\\
  &\leq 20\lvert\delta\rvert\lvert j\rvert^2.
 \end{align*}
Let $\beta_0,\beta_1,\beta_2$ be the roots of $g(z)$. Then for $\beta_i$ the closest root of $g$ to $j(2\tau)$,
\begin{equation*}
 \lvert j(2\tau)-\beta_i\rvert \leq \left( 2.2\lvert\delta\rvert\lvert j\rvert^5\right)^{1/3}\leq \left(2.2\cdot2^{-100}\lvert j\rvert^6\right)^{1/3}\leq10^{-9}\lvert j\rvert^2,
\end{equation*}
and we let $\beta_0$ be the closest root of $g$ to $j(2\tau)$. As $0.98\lvert j\rvert^2\leq j(2\tau)\leq1.02\lvert j\rvert^2$, $0.97\lvert j\rvert^2\leq \lvert\beta_0\rvert\leq1.03\lvert j\rvert^2$. 
For $\tau'=\frac{\tau}{2}$, with $\beta_i$ the closest root of $g$ to $j(\tau')$,
\begin{equation*}
 \lvert j(\tau')-\beta_i\rvert \leq \left( 20\lvert\delta\rvert\lvert j\rvert^2\right)^{1/3}\leq \left(20\cdot2^{-300}\lvert j\rvert^3\right)^{1/3}\leq10^{-29}\lvert j\rvert,
\end{equation*}
and in particular, 
\begin{equation}
 \lvert\beta_i\rvert\leq 10^{-29}\lvert j\rvert+1.17\lvert j\rvert^{1/2}<0.97\lvert j\rvert^2\leq\lvert\beta_0\rvert,
\end{equation}
so $\beta_i\neq \beta_0$. Let this root of $g$ be $\beta_1$. Now we may improve the bound on $\lvert j(\tau')-\beta_1\rvert$,
\begin{equation*}
 \lvert j(\tau')-\beta_1\rvert^2\leq \frac{20\lvert\delta\rvert\lvert j\rvert^2}{\lvert j(\tau')-\beta_0\rvert}\leq\frac{20\lvert\delta\rvert\lvert j\rvert^2}{0.98\lvert j\rvert^2-1.17\lvert j\rvert^{1/2}}\leq30\cdot2^{-300}\lvert j\rvert,
\end{equation*}
so
\begin{equation*}
 \lvert j(\tau')-\beta_1\rvert\leq 10^{-88}\lvert j\rvert^{1/2},
\end{equation*}
and consequently 
\begin{equation*}
 0.82\lvert j\rvert^{1/2}\leq \lvert\beta_1\rvert\leq1.18\lvert j\rvert^{1/2}.
\end{equation*}
We now bound $\beta_2$. By our bound on the difference between the constant coefficients of $f$ and $g$, the constant coefficient $-\beta_0\beta_1\beta_2$ of $g$ is bounded in absolute value by
\begin{equation*}
 \left\lvert j(2\tau)j\left(\frac{\tau}{2}\right)j\left(\frac{\tau+1}{2}\right)\right\rvert + 3.4\delta\lvert j\rvert^2\leq 1.4\lvert j\rvert^3+0.01\lvert j\rvert^3\leq 1.5\lvert j\rvert^3,
\end{equation*}
so that
\begin{equation*}
 \lvert\beta_2\rvert\leq \frac{1.5\lvert j\rvert^{3}}{0.97\lvert j\rvert^2\cdot0.82\lvert j\rvert^{1/2}}\leq 2\lvert j\rvert^{1/2}.
\end{equation*}
Returning to $g(j(2\tau))$, we now bound below the terms $\lvert j(2\tau)-\beta_i\rvert$ for $i=1,2$ in order to improve our inequality for $\lvert j(2\tau)-\beta_0\rvert$. We now have, for $i=1,2$,
\begin{align*}
 \lvert j(2\tau)-\beta_{i}\rvert &\geq \lvert j(2\tau)\rvert - 2\lvert j\rvert^{1/2}\\
 &\geq0.97\lvert j\rvert^2.
\end{align*}
This now improves our bound on the difference of $\beta_0$ to $j(2\tau)$,
\begin{align*}
 \lvert j(2\tau)-\beta_0\rvert &\leq \frac{2.2\delta\lvert j\rvert^5}{\lvert j(2\tau)-\beta_1\rvert\lvert j(2\tau)-\beta_2\rvert} \\
 &\leq \frac{2.2\delta\lvert j\rvert^5}{(0.97\lvert j\rvert^2)^2}\\
 &\leq 2.4\delta\lvert j\rvert.
\end{align*}
So the relative precision of $\beta_0$ as an approximation to $j(2\tau)$ is bounded by
\begin{equation*}
 \frac{\lvert j(2\tau)-\beta_0\rvert}{\lvert j(2\tau)\rvert}\leq \frac{2.4\lvert \delta\rvert\lvert j\rvert}{0.97\lvert j\rvert^2}\leq2.5\frac{\lvert\delta\rvert}{\lvert j\rvert}\leq 2^{-P+2}.
\end{equation*}
\end{proof}

\begin{proposition}
 If $\text{\normalfont Im}(\tau)\geq 3$ and $\tilde{j}$ is an approximation to $j(\tau)$ of relative precision $2^{-P}$, where $P$ is at least $300$, applying Newton's method to $\Phi_2(\tilde{j},z)$, with starting point
 \begin{equation*}
  \tilde{j}^2-2\cdot744\tilde{j}-2\cdot196884+744^2+744,
 \end{equation*}
will obtain an approximation to $j(2\tau)$ of relative precision $2^{-P+3}$ after at most $[2\log P+\log\log\lvert j\rvert]$ steps of Newton iteration.
\end{proposition}
\begin{proof}
We first give a rough approximation to $\beta_0$. Writing
\begin{equation*}
 j(\tau) = e^{-2\pi i\tau}+744+196884e^{2\pi i\tau} + f(\tau),
\end{equation*}
we have
\begin{align*}
 j(\tau)^2-2\cdot744j(\tau)&+2\cdot19688+2\cdot744^2+744 \\
 &= e^{-4\pi i\tau}+744+196884^2e^{4\pi i\tau}+2(196884e^{2\pi i\tau}+e^{-2\pi i\tau})f(\tau)+f(\tau)^2
\end{align*}
so that, as $\text{\normalfont Im}(\tau)\geq 3$, $f$ is maximized when $\text{\normalfont Re}(\tau)=0$, and $f(\text{\normalfont Im}(\tau))$ is decreasing in $\text{\normalfont Im}(\tau)$,
\begin{align*}
 \lvert j(2\tau) &- (j(\tau)^2-744j(\tau)+19688-744^2+744)\rvert\\
 &\leq (196884+196884^2)e^{-4\pi\text{\normalfont Im}(\tau)}+\lvert f(\text{\normalfont Im}(2\tau))\rvert\\
 &\quad\quad+2(196884e^{2\pi \text{\normalfont Im}(\tau)}+e^{2\pi \text{\normalfont Im}(\tau)})\lvert f(\text{\normalfont Im}(\tau))\rvert+\lvert f(\text{\normalfont Im}(\tau))\rvert^2\\
 &\leq 10^{-7}+\lvert f(6)\rvert+10^6\lvert f(3)\rvert+\lvert f(3)\rvert^2\\
 &\leq 0.4
\end{align*}
We now let 
\begin{equation*}
 \gamma_0 = \tilde{j}^2-744\tilde{j}+19688-744^2+744.
\end{equation*}
This will be our starting point for Newton iteration to find $\beta_0$. We now bound the terms appearing in Kantorovich's theorem. Firstly,
\begin{align*}
\Phi_2(\tilde{j},z)=z^3&+(-\tilde{j}^2+1488\tilde{j}-162000)z^2\\
  &+(1488\tilde{j}^2+40773375\tilde{j}+8748000000)z\\
  &+\tilde{j}^3-162000\tilde{j}^2+8748000000\tilde{j}-157464000000000.
 \end{align*}
Now it is clear, as $\lvert j-\tilde{j}\rvert\leq2^{-300}\lvert j\rvert$, that $\lvert \gamma_0-j(2\tau)\rvert\leq 0.4+\frac{\lvert j\rvert^2}{2^{90}}$, and so as $0.98\lvert j\rvert^2\leq\lvert j(2\tau)\rvert\leq1.02\lvert j\rvert^2$, we have $0.979\lvert j\rvert^2\leq\lvert \gamma_0\rvert\leq1.021\lvert j\rvert^2$. we will take the $r$ of Kantorovich's theorem to be $0.009\lvert j\rvert^2$, and give bounds for the disc $\lvert\gamma-\gamma_0\rvert\leq0.009\lvert j\rvert^2$. For $\gamma$ in this disc, $0.97\lvert j\rvert^2\leq\lvert \gamma\rvert\leq1.03\lvert j\rvert^2$ and in addition $\lvert\tilde{j}\rvert\leq1.01\lvert j\rvert$, so for an upper bound on the first derivative, we have
\begin{align*}
 \lvert\Phi_2'(\tilde{j},\gamma)\rvert&\leq 3\lvert\gamma\rvert^2+2(\lvert\tilde{j}\rvert^2+1488\lvert\tilde{j}\rvert+162000)\lvert\gamma_0\rvert\\
  &\quad+(1488\lvert\tilde{j}\rvert^2+40773375\lvert\tilde{j}\rvert+8748000000)\\
  &\leq 3.2\lvert j\rvert^4+2.2\lvert j\rvert^4+10^{-7}\lvert j\rvert^4\\
  &\leq 5.5\lvert j\rvert^4,
\end{align*}
and for a lower bound, we have
\begin{align*}
 \lvert\Phi_2'(\tilde{j},\gamma)\rvert&\geq 3\lvert\gamma\rvert^2-2(\lvert\tilde{j}\rvert^2+1488\lvert\tilde{j}\rvert+162000)\lvert\gamma_0\rvert\\
  &\quad-(1488\lvert\tilde{j}\rvert^2+40773375\lvert\tilde{j}\rvert+8748000000)\\
  &\geq 2.82\lvert j\rvert^4-2.1\lvert j\rvert^4-10^{-7}\lvert j\rvert^4\\
  &\geq 0.71\lvert j\rvert^4.
\end{align*}
For our upper bound on the function at $\gamma_0$, as
 $\lvert\beta_1\rvert, \lvert\beta_2\rvert\leq2\lvert j\rvert^{1/2}$, and as $\lvert\beta_0-j(2\tau)\rvert\leq 2^{-298}\lvert j\rvert^2$ we have $\lvert\gamma_0-\beta_0\rvert\leq0.4+2^{-298}\lvert j\rvert^2$, so
\begin{align*}
 \lvert\Phi_2(\tilde{j},\gamma_0)\rvert&=\lvert\gamma_0-\beta_0\rvert\lvert\gamma_0-\beta_1\rvert\lvert\gamma_0-\beta_2\rvert\\
 &\leq(0.4+2^{-289}\lvert j\rvert^2)(1.03\lvert j\rvert^2+2\lvert j\rvert^{1/2})^2\\
 &\leq 2^{-54}\lvert j\rvert^6.
\end{align*}
For the second derivative, we have for an upper bound
\begin{align*}
 \lvert\Phi_2''(\tilde{j},\gamma)\rvert&\leq6\lvert\gamma\rvert+2\lvert \tilde{j}\rvert^2+1488\lvert \tilde{j}\rvert+162000\\
 &\leq 8.3\lvert j\rvert^2,
\end{align*}
and for a lower bound,
\begin{align*}
 \lvert\Phi_2''(\tilde{j},\gamma_0)\rvert&\geq6\lvert\gamma\rvert-2\lvert \tilde{j}\rvert^2-1488\lvert \tilde{j}\rvert-162000\\
 &\geq 3.8\lvert j\rvert^2.
\end{align*}
Now, by Theorem \ref{kantorovich}, as
\begin{equation*}
 \frac{\lvert\Phi_2(\tilde{j},\gamma_0)\rvert\lvert\Phi_2''(\tilde{j},\gamma)\rvert}{\lvert\Phi_2'(\tilde{j},\gamma)\rvert^2}\leq\frac{2^{-54}\lvert j\rvert^6\cdot8.3\lvert j\rvert^2}{(0.71\lvert j\rvert^4)^2}\leq2^{-49}<\frac{1}{2},
\end{equation*}
and
\begin{align*}
2\eta &= 2\frac{\lvert\Phi_2(\tilde{j},\gamma_0)\rvert}{\lvert\Phi_2'(\tilde{j},\gamma_0)\rvert}\leq \frac{2^{-54}\lvert j\rvert^6}{0.71\lvert j\rvert^4}\leq0.001\lvert j\rvert^2,\\
 r&=0.009\lvert j\rvert^2,
\end{align*}
Newton's method will converge to the root $\beta_0$, with a rate of convergence
\begin{equation*}
 \lvert \gamma_k - \beta_0\rvert\leq\frac{1}{2^k}2^{-48\cdot2^{k}}\frac{\lvert\Phi_2'(\tilde{j},\gamma_0)\rvert}{\lvert\Phi_2''(\tilde{j},\gamma_0)\rvert}\leq\frac{1}{2^k}2^{-48\cdot2^k}\frac{5.5\lvert j\rvert^4}{3.8\lvert j\rvert^2}\leq2^{-48\cdot2^k}\lvert j\rvert^2
\end{equation*}
for $k\geq 1$. In particular, when $k\geq [2\log P+2\log\log\lvert j\rvert]$, $\lvert \gamma_k-\beta_0\rvert\leq 2^{-P}$. Now as, by Proposition \ref{phi2discrepancy}, $\beta_0$ is an approximation to $j(2\tau)$ of relative precision $2^{-P+2}$, it may be easily shown that after $[2\log P+2\log\log\lvert j\rvert]$ steps, $\gamma_k$ will be an approximation to $j(2\tau)$ of relative precision $2^{-P+3}$.
\end{proof}

\begin{lemma}
 Suppose that $\lvert j(\tau)\rvert\geq 2^{P+12}$, and $\tilde{j}$ is an approximation to $j(\tau)$ of relative precision at least $2^{-P}$. Then
 \begin{equation*}
  -\frac{\log\tilde{j}}{2\pi}
 \end{equation*}
is an approximation to $\tau$ of absolute precision $2^{-P}$.
\end{lemma}
\begin{proof}
 Let $j(\tau)+\delta = \tilde{j}$. As
\begin{equation*}
 \lvert j(\tau)-e^{-2\pi\tau}\rvert \leq 2079,
\end{equation*}
we have
\begin{equation*}
 \log(\tilde{j}) = -2\pi\tau+ \log\left(1+\frac{\delta+2079\theta}{j(\tau)}\right),
\end{equation*}
where $\lvert\theta\rvert\leq1$. So as $\lvert\log(1+z)\rvert\leq1.1\lvert z\rvert$ when $\lvert z\rvert\leq 0.05$,
\begin{equation*}
 \lvert\log\tilde{j}+2\pi \tau\rvert \leq 1.1\cdot2^{-P}+0.6\cdot2^{-P}\leq 1.7\cdot2^{-P}.
\end{equation*}
So that 
\begin{equation*}
 \left\lvert\tau-\left(-\frac{\log\tilde{j}}{2\pi}\right)\right\rvert\leq \frac{1.7}{2\pi}2^{-P}\leq 2^{-P}.
\end{equation*}
\end{proof}

Now the algorithm to invert $j$ when $\text{\normalfont Im}(\tau)\geq 3$ proceeds as follows -- if $\lvert j\rvert\leq 2^{P+12}$, first iteratively compute approximations to $j(2^k\tau)$ by Newton's method applied to $\Phi_2(\tilde{j},z)$, up to $k = \left[2\log\left(\frac{P+12}{\text{\normalfont Im}(\tau)}\right)\right]+1$. Then calculate the logarithm of $j(2^k\tau)$ to relative precision $2^{-P-2}$ (note that $\lvert \tau\rvert\geq 1$), and divide by $-2\pi$, where we have calculated $2\pi$ to relative precision $2^{-P-2}$. This process entails a loss of precision of at most $5\left[2\log\left(\frac{P+12}{\text{\normalfont Im}(\tau)}\right)\right]+2$, which is bounded by $11\log P$ when $P\geq 400$, and the precision at all applications of Newton's method is at least $2^{-300}$, so our assumptions on the precision of our approximations in the propositions of this section are satisfied at each application. As the computational complexity of Newton's method is $O(M(P)\log P)$, and of the complex logarithm and computing $\pi$ are $O(M(P)(\log P)^2)$, if $\lvert j\rvert\leq2^{P+12}$, the algorithm has time complexity
\begin{equation*}
 O\left(M(P)\left[2\log\left(\frac{P+12}{\text{\normalfont Im}(\tau)}\right)+1\right][2\log P+2\log\log\lvert j\rvert]\right) = O\left(M(P)(\log P)^2\right),
\end{equation*}
and if $\lvert j\rvert\geq 2^{P+12}$, time complexity
\begin{equation*}
 O\left(M(P)(\log P)^2\right),
\end{equation*}
where the implicit constants are not too large and could be made effective. 

\subsection{Near 1728 and 0}
When $\tau$ is close to either $i$ or $\frac{\pm1+i\sqrt{3}}{2}$, we will make use of the modular polynomial $\Phi_2(X,Y)$ to obtain an approximation to one of $j\left(2\tau\right)$, $j\left( \frac{\tau}{2}\right)$, or $j\left(\frac{\tau+1}{2}\right)$, for which the SL$_2(\mathbb{Z})$--equivalent elements of $\mathcal{F}$ to either $2\tau$, $\frac{\tau}{2}$, or $\frac{\tau+1}{2}$ will lie in the aforementioned compact set, to which we may apply Newton's method. We carry out the analysis only for $i$ and $\frac{1+i\sqrt{3}}{2}$, as when $\left\lvert\tau-\frac{-1+i\sqrt{3}}{2}\right\rvert\leq2^{-31}$, the SL$_2(\mathbb{Z})$--equivalent $\tau+1$ satisfies $\left\lvert(\tau+1)-\frac{1+i\sqrt{3}}{2}\right\rvert\leq2^{-31}$.

\begin{lemma}\label{taylor_series_at_i}
 If $\lvert\delta\rvert\leq 2^{-30}$ then 
  \begin{equation*}
\left\lvert j(i+\delta)-1728-\frac{j^{(2)}(i)}{2}\delta^2\right\rvert \leq 0.07\lvert \delta\rvert^2,
\end{equation*}
and
\begin{equation*}
\left\lvert j\left(\frac{1+i\sqrt{3}}{2}+\delta\right)-\frac{j^{(3)}\left(\frac{1+i\sqrt{3}}{2}\right)}{3!}\delta^3\right\rvert \leq 0.07\lvert \delta\rvert^3.
\end{equation*}
\end{lemma}
\begin{proof}
We first bound the coefficients of the Taylor series of $j$ at $z=i$ and $z=\frac{1+i\sqrt{3}}{2}$. Firstly, by Theorem 1 of \cite{jcoefficients}, the coefficient of $e^{2\pi in\tau}$ in the $q$--expansion of $j$ is bounded by $4e^{4\pi\sqrt{n}}$, so the corresponding coefficient of the $k$'th derivative is bounded by $(2\pi n)^ke^{4\pi\sqrt{n}}$. Further, for $n\geq 1$, $e^{4\pi\sqrt{n}-\sqrt{3}\pi n}\leq 100e^{-\pi n}$, so we have the following bound of the derivatives at these two points,
\begin{align*}
 \lvert j^{(k)}(i)\rvert, \left\lvert j^{(k)}\left(\frac{1+i\sqrt{3}}{2}\right)\right\rvert&\leq (2\pi)^ke^{2\pi}+744+400\sum_{n=1}^\infty(2\pi n)^ke^{-\pi n}.
\end{align*}
We now bound the sum in this expression. Firstly, we have the identity
\begin{equation*}
 \sum_{n=1}^\infty e^{-\pi nx} = \frac{1}{e^{\pi x}-1},
\end{equation*}
and so consider the derivatives of this function, expressing the derivative as sums of the form
\begin{equation*}
 \sum \alpha_i\frac{e^{a\pi x}}{(e^{\pi x}-1)^b},
\end{equation*}
where each term occurring in the expression for the $k$'th derivative is derived from taking the derivative of either the numerator or denominator of a term occurring in the $k-1$'th derivative, i.e. there is no collection of terms with $(a,b)$ equal. It is clear that there are at most $2^k$ terms, and, passing from one derivative to the next, $\lvert\alpha_i\rvert$ may increase by at most $\pi\max\{a,b\}$, and that $a\leq b\leq k+1$, and $a\leq k$. So $\lvert\alpha_i\rvert\leq \pi^k(k+1)!$, and a bound for the whole expression evaluated at $1$ is therefore
\begin{equation*}
 2^k\pi^k(k+1)!\frac{e^{k\pi}}{(e^{\pi}-1)^k}\leq 14^kk!.
\end{equation*}
Now we have the bounds
\begin{align*}
 \lvert j^{(k)}(i)\rvert, \left\lvert j^{(k)}\left(\frac{1+i\sqrt{3}}{2}\right)\right\rvert&\leq (2\pi)^ke^{2\pi}+744+400\cdot14^kk!\\
 &\leq1700\cdot14^kk!.
\end{align*}
At $z=i$ and $z=\frac{1+i\sqrt{3}}{2}$, the Taylor series expansions for $j(z)$ are
\begin{align*}
 j(z) = 1728 + \sum_{n=2}^{\infty}c_n(z-i)^n,\\
 j(z) = \sum_{n=3}^{\infty}c'_n\left(z-\frac{1+i\sqrt{3}}{2}\right)^n.
\end{align*}
where $c_n$ and $c_n'$ are bounded in absolute value by $1400\cdot14^k$.
 We now bound the tails of the sums of the Taylor series,
 \begin{equation*}
  \sum_{n=3}^\infty \lvert c_n\rvert\lvert\delta\rvert^n\leq \lvert\delta\rvert^3\sum_{n=0}^\infty1700\cdot14^{n+3}\lvert\delta\rvert^n,
 \end{equation*}
and 
\begin{equation*}
  \sum_{n=4}^\infty \lvert c_n'\rvert\lvert\delta\rvert^n\leq \lvert\delta\rvert^4\sum_{n=0}^\infty1700\cdot14^{n+4}\lvert\delta\rvert^n.
 \end{equation*}
 Now as $\lvert\delta\rvert\leq 2^{-30}$,
 \begin{equation*}
  \sum_{n=0}^\infty1700\cdot14^{n+4}\lvert\delta\rvert^n=\frac{1700\cdot14^4}{1-14\lvert\delta\rvert}\leq 7\cdot10^7,
 \end{equation*}
so that
 \begin{align*}
\left\lvert j(i+\delta)-1728-\frac{j^{(2)}(i)}{2}\delta^2\right\rvert &\leq 2^{-30}\cdot7\cdot10^7\lvert\delta\rvert^2\leq 0.07\lvert \delta\rvert^2,\\
\left\lvert j\left(\frac{1+i\sqrt{3}}{2}+\delta\right)-\frac{j^{(3)}\left(\frac{1+i\sqrt{3}}{2}\right)}{3!}\delta^3\right\rvert &\leq 2^{-30}\cdot7\cdot10^7\lvert\delta\rvert^3\leq 0.07\lvert \delta\rvert^3.
 \end{align*}
\end{proof}
Furthermore, we note that
$49700\geq\lvert j^{(2)}(i)\rvert\geq 49600$, and $275000\geq\left\lvert j^{(3)}\left(\frac{1+i\sqrt{3}}{2}\right)\right\rvert\geq 274000$, so if $\lvert\tau - i\rvert\leq 2^{-30}$, by the previous lemma, we have
\begin{equation*}
2.4\cdot10^4\lvert\delta\rvert^2\leq \lvert j-1728\rvert\leq2.5\cdot10^4\lvert\delta\rvert,
\end{equation*}
and if $\left\lvert\tau - \frac{1+i\sqrt{3}}{2}\right\rvert\leq 2^{-30}$,
\begin{equation*}
 4.5\cdot10^4\lvert\delta\rvert^3\leq\lvert j\rvert\leq4.6\cdot10^4\lvert\delta\rvert^3,
\end{equation*}
from which we deduce the following,
\begin{lemma}\label{j_to_tau_small}
If $\lvert\tau-i\rvert\leq 2^{-30}$ and $P\geq 300$, then:
\begin{enumerate}
 \item If $\lvert j(\tau) - 1728\rvert\leq 2^{-P}$, then $\lvert\tau - i\rvert\leq 2^{-P/2-7}$.
 \item If $\lvert j(\tau) - 1728\rvert\geq 2^{-P}$, then $\lvert\tau - i\rvert\geq 2^{-P/2-8}$.
 \end{enumerate}
 If $\left\lvert\tau-\frac{1+i\sqrt{3}}{2}\right\rvert\leq2^{-30}$ and $P\geq 300$, then:
 \begin{enumerate}
 \item If $\lvert j\rvert\leq 2^{-P}$, then $\left\lvert\tau - \frac{1+i\sqrt{3}}{2}\right\rvert\leq2^{-P/3-5}$.
 \item If $\lvert j\rvert\geq 2^{-P}$, then $\left\lvert\tau - \frac{1+i\sqrt{3}}{2}\right\rvert\geq2^{-P/3-6}$.
\end{enumerate}
\end{lemma}

\begin{lemma}\label{taylor_series_at_2i}
 Suppose that $\lvert\delta\rvert\leq 2^{-28}$. Then
 \begin{equation*}
\left\lvert j\left(2i+\delta\right)-j\left(2i\right)-j'\left(2i\right)\delta\right\rvert \leq 0.2\lvert \delta\rvert,
\end{equation*}
and 
\begin{equation*}
\left\lvert j\left(i\sqrt{3}+\delta\right)-j\left(i\sqrt{3}\right)-j'\left(i\sqrt{3}\right)\delta\right\rvert \leq 0.2\lvert \delta\rvert.
\end{equation*}
\end{lemma}
\begin{proof}
We proceed similarly to the previous lemma -- as $17e^{-2\pi n}\geq e^{4\pi n-2\sqrt{3}\pi n}$, 
\begin{equation*}
 \lvert j^{(k)}(2i)\rvert,\lvert j^{(k)}(\sqrt{3}i)\rvert \leq (2\pi)^ke^{4\pi}+744+68\sum_{n=1}^{\infty}(2\pi n)^ke^{-2\pi n},
\end{equation*}
and a similar analysis to the previous lemma bounds the sum in this inequality by
\begin{equation}
 (2\pi)^ke^{4\pi}+744+68\cdot 13^kk! \leq 300000\cdot13^kk!.
\end{equation}
 So when $\lvert\delta\rvert\leq 2^{-28}$,
 \begin{align*}
\left\lvert j\left(2i+\delta\right)-j(2i)-j'\left(2i\right)\delta\right\rvert &\leq \sum_{n=2}^{\infty}300000\cdot13^n\delta^n\\
&\leq \sum_{n=2}^{\infty}300000\cdot13^n\delta^n\\
&\leq 2^{-28}\cdot300000\cdot13^2\lvert\delta\rvert\sum_{n=0}^{\infty}13^n2^{-28n}\\
&\leq0.2\lvert \delta\rvert.
\end{align*}
 and as our bound for the terms of the Taylor series applies to both expansions, we similarly have
 \begin{align*}
\left\lvert j\left(i\sqrt{3}+\delta\right)-j\left(i\sqrt{3}\right)-j'\left(i\sqrt{3}\right)\delta\right\rvert \leq 0.2\lvert \delta\rvert.
\end{align*}
\end{proof}

We now give two lemmas on the separation and closeness of the three preimages of roots of the modular polynomial $\Phi_2(j(\tau),z)$, for $\tau$ near $i$ and $\frac{1+i\sqrt{3}}{2}$.
\begin{lemma}\label{preimage_closeness}
 Suppose that $\tau\in\mathcal{F}$. Then, if $\tau = i+\delta$,
 \begin{align*}
  \lvert 2\tau - 2i\rvert&\leq 2\lvert\delta\rvert,\\
  \left\lvert -\frac{2}{\tau}- 2i\right\rvert&\leq 2\lvert\delta\rvert,
 \end{align*}
and if $\tau = \frac{1+i\sqrt{3}}{2}+\delta$,
\begin{align*}
  \lvert \left(2\tau-1\right) - \sqrt{3}i\rvert&\leq 2\lvert\delta\rvert,\\
  \left\lvert \left(1-\frac{2}{\tau}\right)- \sqrt{3}i\right\rvert&\leq 2\lvert\delta\rvert,\\
  \left\lvert \left(-1-\frac{2}{\tau-1}\right)- \sqrt{3}i\right\rvert&\leq 2\lvert\delta\rvert.
 \end{align*}
\end{lemma}
\begin{proof}
First note that as $\tau\in\mathcal{F}$, $\lvert\tau\rvert\geq 1$. Let $2\tau = 2i+\delta_1$, and $-\frac{2}{\tau}=2i+\delta_2$. For $\delta_1$, $\lvert2\tau-2i\rvert = 2\lvert \delta\rvert$, and for $\delta_2$,
\begin{align*}
 \lvert\delta_2\rvert=\left\lvert -\frac{2}{\tau}-2i\right\rvert &= \left\lvert \frac{-2-2i\tau}{\tau}\right\rvert\\
 &=\frac{\lvert2i\delta\rvert}{\lvert\tau\rvert}\\
 &\leq 2\lvert\delta\rvert
\end{align*}
If $\tau = \frac{1+i\sqrt{3}}{2}+\delta$, let $2\tau-1 = \sqrt{3}i+ \delta_1$, $\left(1-\frac{2}{\tau}\right) = \sqrt{3}i+\delta_2$, and $\left(1-\frac{2}{\tau-1}\right)=\sqrt{3}i+\delta_3$. For $\delta_1$,
\begin{align*}
 \lvert\delta_1\rvert = \left\lvert 2\tau-1-\sqrt{3}i\right\rvert\leq 2\lvert\delta\rvert,
\end{align*}
for $\delta_2$,
\begin{align*}
 \lvert\delta_2\rvert = \left\lvert \left(1-\frac{2}{\tau}\right)- \sqrt{3}i\right\rvert& = \left\lvert\frac{\tau-2-\sqrt{3}i\tau}{\tau}\right\rvert\\
 &=\left\lvert\frac{(1-\sqrt{3}i)\delta}{\tau}\right\rvert\\
 &\leq2\lvert\delta\rvert,
\end{align*}
and for $\delta_3$,
\begin{align*}
 \lvert\delta_3\rvert = \left\lvert \left(-1-\frac{2}{\tau-1}\right)- \sqrt{3}i\right\rvert &= \left\lvert\frac{-\tau+1-2-\sqrt{3}i\tau+\sqrt{3}i}{\tau}\right\rvert\\
 &=\frac{\lvert(1+\sqrt{3}i)\delta\rvert}{\lvert\tau-1\rvert}\\
 &\leq 2\lvert\delta\rvert.
\end{align*}
\end{proof}

\begin{lemma}\label{preimage_separation}
 If $\tau = i+\delta$, $\tau \in \mathcal{F}$, with $\lvert \delta\rvert \leq 2^{-30}$, then the distance between $2\tau$ and $-\frac{2}{\tau}$ is at least $3.99\lvert\delta\rvert$, and the distance between either of these and the point in $\mathcal{F}$ which is $\text{\normalfont SL}_2(\mathbb{Z})$--equivalent to $\frac{\tau+1}{2}$ is at least $0.99$ in magnitude. If $\tau = \frac{1+i\sqrt{3}}{2}+\delta$, $\tau\in\mathcal{F}$, then the distance between any pair of the three points lying in $\mathcal{F}$ which are $\text{\normalfont SL}_2(\mathbb{Z})$--equivalent to $2\tau$, $\frac{\tau}{2}$, and $\frac{\tau+1}{2}$ is at least $3.46\lvert\delta\rvert$.
\end{lemma}
\begin{proof}
 Firstly, for the distance between $\tau$ and $-\frac{2}{\tau}$, we have
  \begin{align*}
   2\tau +\frac{2}{\tau} &= \frac{2(\tau^2+1)}{\tau}\\
   &=\frac{4i\delta+2\delta^2}{\tau}.
  \end{align*}
  As $\lvert\delta\rvert\leq 2^{-30}$, $\lvert\delta\rvert^2\leq2^{-30}\lvert\delta\rvert$, and $\lvert i+\delta\rvert\leq1+2^{-30}$, so
  \begin{equation*}
   \left\lvert2\tau-\left(-\frac{2}{\tau}\right)\right\rvert\geq3.99\lvert\delta\rvert.
  \end{equation*}
  Now for $-\frac{2}{\tau+1}+1$, we have
  \begin{align*}
   -\frac{2}{\tau+1}+1-i &= \frac{-2+\tau+1-i\tau-i}{\tau+1}\\
   &=\frac{\delta-i\delta}{1+i+\delta}
  \end{align*}
  so that 
\begin{equation}
 \left\lvert\left(-\frac{2}{\tau+1}+1\right)-i\right\rvert \leq 2\lvert\delta\rvert
\end{equation}
and so either $-\frac{2}{\tau+1}+1$ or $-\left(-\frac{2}{\tau+1}+1\right)^{-1}$ is $\text{\normalfont SL}_2(\mathbb{Z})$--equivalent to $\frac{1+\tau}{2}$ and in $\mathcal{F}$, and its distance from $i$ is bounded by $\left\lvert 1+\frac{2\lvert\delta\rvert}{1-2\lvert\delta\rvert}\right\rvert\leq 2^{-28}$. By Lemma \ref{preimage_closeness}, the distance of either $2\tau$ or $-\frac{2}{\tau}$ from $2i$ is at most $2\lvert\delta\rvert\leq2\cdot2^{-30}$, so both of these are separated from either $-\frac{2}{\tau+1}+1$ or $-\left(-\frac{2}{\tau+1}+1\right)^{-1}$ by a distance of at least $0.99$.

Now for $\tau = \frac{1+i\sqrt{3}}{2}+\delta$, the $\text{\normalfont SL}_2(\mathbb{Z})$--equivalent points to $2\tau$, $\frac{\tau}{2}$, and $\frac{\tau+1}{2}$ are $2\tau-1$, $1-\frac{2}{\tau}$, and $-1-\frac{2}{\tau-1}$ respectively. For their differences, we have,
 \begin{align*}
   2\tau-1 -\left(1-\frac{2}{\tau}\right) &= \frac{2(\tau^2-\tau+1)}{\tau}\\
   &=\frac{i2\sqrt{3}\delta+2\delta^2}{\tau}.
  \end{align*}
  So as $\lvert \tau\rvert\leq1+2^{-30}$,
  \begin{equation*}
   \left\lvert 2\tau-1 -\left(1-\frac{2}{\tau}\right)\right\rvert\geq 3.46\lvert \delta\rvert
  \end{equation*}
  For the first and third points,
  \begin{align*}
   2\tau-1 -\left(-1-\frac{2}{\tau-1}\right) &= \frac{2(\tau^2-\tau+1)}{\tau-1}\\
   &=\frac{i2\sqrt{3}\delta+2\delta^2}{\tau-1}.
  \end{align*}
As $\lvert\tau-1\rvert\leq1+2^{-30}$, we have the lower bound
\begin{equation*}
  \left\lvert2\tau-1 -\left(-1-\frac{2}{\tau-1}\right)\right\rvert\geq3.46\lvert \delta\rvert.
\end{equation*}
For the second and third points,
\begin{align*}
   \left(1-\frac{2}{\tau}\right) -\left(-1-\frac{2}{\tau-1}\right) &= \frac{2(\tau^2-\tau+1)}{\tau(\tau-1)}\\
   &=\frac{i2\sqrt{3}\delta+2\delta^2}{\tau(\tau-1)}.
\end{align*}
So we obtain the lower bound
\begin{equation*}
  \left\lvert\left(1-\frac{2}{\tau}\right) -\left(-1-\frac{2}{\tau-1}\right)\right\rvert\geq3.46\lvert \delta\rvert.
\end{equation*}
\end{proof}

\begin{lemma}\label{bounds_for_images}
 If $\tau = i+\delta$, with $\lvert\delta\rvert\leq2^{-30}$, then 
 \begin{equation*}
  7.18\cdot10^{6}\lvert\delta\rvert\leq\left\lvert j(2\tau)-j\left(\frac{\tau}{2}\right)\right\rvert\leq7.25\cdot10^6\lvert\delta\rvert ,
 \end{equation*}
 and 
 \begin{equation*}
   2.4\cdot10^5\leq\left\lvert j(2\tau)-j\left(\frac{\tau+1}{2}\right)\right\rvert,  \left\lvert j\left(\frac{\tau}{2}\right)-j\left(\frac{\tau+1}{2}\right)\right\rvert\leq3.1\cdot10^5,
 \end{equation*}
 and if $\tau = \frac{1+i\sqrt{3}}{2}+\delta$, with $\lvert\delta\rvert\leq 2^{-30}$, then for $\tau_1,\tau_2\in\{2\tau,\frac{\tau}{2},\frac{\tau+1}{2}\}$, $\tau_1\neq\tau_2$,
 \begin{equation*}
  1.15\cdot10^6\lvert\delta\rvert\leq\lvert j(\tau_1)-j(\tau_2)\rvert\leq 1.34\cdot10^6\lvert\delta\rvert.
 \end{equation*}
\end{lemma}
\begin{proof}
 We first use the Taylor series expansion of $j(z)$ at $2i$. By Lemma \ref{preimage_closeness}, we may apply Lemma \ref{taylor_series_at_2i}, and using also Lemma \ref{preimage_separation}, we obtain the lower bound
 \begin{align*}
  \left\lvert j(2\tau)-j\left(-\frac{2}{\tau}\right)\right\rvert&\geq \left\lvert j(2i)+j'(2i)\delta_1-j(2i)-j'(2i)\delta_2\right\rvert - 0.4\lvert\delta\rvert\\
  &=\left\lvert j'(2i)\right\rvert\left\lvert \delta_1- \delta_2\right\rvert - 0.4\lvert\delta\rvert\\
  &=\left\lvert j'(2i)\right\rvert\left\lvert 2\tau+\frac{2}{\tau}\right\rvert - 0.4\lvert\delta\rvert\\
  &\geq 1.8\cdot10^6\cdot3.99\rvert\delta\lvert - 0.4\lvert\delta\rvert\\
  &\geq 7.18\cdot10^{6}\lvert\delta\rvert.
 \end{align*}
 Similarly we have an upper bound of 
 \begin{equation*}
  \left\lvert j(2\tau)-j\left(-\frac{2}{\tau}\right)\right\rvert\leq1.81\cdot10^6\cdot4\rvert\delta\lvert+ 194\lvert\delta\rvert\leq7.25\cdot10^6\lvert\delta\rvert.
 \end{equation*}
 By Lemma \ref{preimage_separation}, the $\text{\normalfont SL}_2(\mathbb{Z})$--equivalent point in $\mathcal{F}$ to $\frac{\tau+1}{2}$ is at most $2\lvert\delta\rvert\leq2^{-29}$ from $i$, so 
 \begin{equation*}
  \left\lvert j\left(\frac{\tau+1}{2}\right)\right\rvert\leq e^{2.01\pi}+2079\leq e^{7.9},
 \end{equation*}
 and by Lemma \ref{preimage_closeness}, $2\tau,-\frac{\tau}{2}$ are at most $2^{-29}$ from $2i$,
 \begin{equation*}
  \left\lvert j\left(2\tau\right)\right\rvert,\left\lvert j\left(\frac{\tau}{2}\right)\right\rvert\geq e^{3.99\pi}-2079\geq e^{12.5},
 \end{equation*}
 so that 
 \begin{align*}
  3.1\cdot10^5\geq\left\lvert j(2\tau)-j\left(\frac{\tau+1}{2}\right)\right\rvert,  \left\lvert j\left(-\frac{\tau}{2}\right)-j\left(\frac{\tau+1}{2}\right)\right\rvert\geq 2.4\cdot10^5.
 \end{align*}
Now using the Taylor series of $j(z)$ at $i\sqrt{3}$, by Lemma \ref{preimage_closeness}, we may apply Lemma \ref{taylor_series_at_2i}, and using also Lemma \ref{preimage_separation}, with $\tau_1$, $\tau_2$ any distinct pair of $2\tau$ ($\text{\normalfont SL}_2(\mathbb{Z})$--equivalent to $2\tau-1$), $\frac{\tau}{2}$, ($\text{\normalfont SL}_2(\mathbb{Z})$--equivalent to $1-\frac{2}{\tau}$), and $\frac{\tau+1}{2}$ ($\text{\normalfont SL}_2(\mathbb{Z})$--equivalent to $-1-\frac{1}{\tau-1}$), with $\tau_j = i\sqrt{3}+\delta_j$,
\begin{align*}
 \left\lvert j(\tau_1)-j(\tau_2)\right\rvert &\geq \left\lvert j(\sqrt{3}i)+j'(\sqrt{3}i)\delta_1-(j(\sqrt{3}i)+j'(\sqrt{3}i)\delta_2)\right\rvert - 0.4\lvert\delta\rvert\\
 &=\lvert j'(\sqrt{3}i)\rvert\lvert\tau_1-\tau_2\rvert-0.4\lvert\delta\rvert\\
 &\geq 334500\cdot3.46\lvert\delta\rvert-0.4\lvert\delta\rvert\\
 &\geq1.15\cdot10^6\lvert\delta\rvert,
\end{align*}
and similarly for an upper bound we have
\begin{align*}
 \left\lvert j(\tau_1)-j(\tau_2)\right\rvert &\leq \left\lvert j(\sqrt{3}i)+j'(\sqrt{3}i)\delta_1-(j(\sqrt{3}i)+j'(\sqrt{3}i)\delta_2)\right\rvert + 0.4\lvert\delta\rvert\\
 &=\lvert j'(\sqrt{3}i)\rvert\lvert\tau_1-\tau_2\rvert+0.4\lvert\delta\rvert\\
 &\leq 334600\cdot4\lvert\delta\rvert+0.4\lvert\delta\rvert\\
 &\leq1.34\cdot10^6\lvert\delta\rvert.
\end{align*}
\end{proof}

\subsubsection{$j$ near $1728$}
We now bound the discrepancy between the roots of $\Phi_2(j,z)$ and $\Phi_2(\tilde{j},z)$ when $\lvert\tau-i\rvert\leq2^{-20}$ and $\lvert \tilde{j}-1728\rvert\geq 2^{-P/3}$.
\begin{lemma}\label{phi2discrepancy_points_1728}
 Suppose that $\tilde{j}$ is an approximation to $j(\tau)$ of relative precision $2^{-P}$, with $P\geq 300$, $\lvert\tau-i\rvert\leq2^{-30}$, and $\left\lvert \tilde{j}-1728\right\rvert\geq 2^{-P/3}$. Then the relative precision of any root of $\Phi_2(\tilde{j},z)$ to the closest root of $\Phi_2(j,z)$ is at least $2^{-P/3+10}$
\end{lemma}
\begin{proof}
Firstly, as in the proof of Lemma \ref{phi2discrepancy}, with $f(z) = \Phi_2(j,z)$ and $g(z) = \Phi_2(\tilde{j},z)$, and $\tilde{j}=j+\delta$, we will bound the size of the coefficients of $f-g$. We note that $\lvert j\rvert\geq1700$. For $z^2$, we have
\begin{equation*}
 \lvert 2j\delta+\delta^2+1448\delta\rvert\leq 3\lvert\delta\rvert\lvert j\rvert,
\end{equation*}
for $z$, we have
\begin{equation*}
 \lvert 2976\delta j+1488\delta^2+40773375\delta\rvert\\
  \leq 30000\lvert\delta\rvert\lvert j\rvert,
\end{equation*}
and for the constant term
\begin{equation*}
  \left\lvert 3\delta j^2+3\delta^2j+\delta^3-16200\delta j-16200\delta^2+8748000000\delta\right\rvert\leq 3100\lvert\delta\rvert\lvert j\rvert^2.
 \end{equation*}
Now evaluating $g$ at $j(2\tau)$, we have, noting that $\lvert j(2\tau)\rvert\leq 0.1\lvert j\rvert^2$,
\begin{align*}
 \lvert g(j(2\tau))\rvert &= \lvert f(j(2\tau))-g(j(2\tau))\rvert\\
 &\leq 3\lvert\delta\rvert\lvert j\rvert(0.1\lvert j\rvert^2)^2+30000\lvert\delta\rvert\lvert j\rvert(0.1\lvert j\rvert^2)+3100\lvert\delta\rvert\lvert j\rvert^2\\
 &\leq 0.04\lvert j\rvert^5
\end{align*}
Letting $\beta_0,\beta_1,\beta_2$ be the roots of $g$, where $\beta_0$ is the closest root of $g(z)$ to $j(2\tau)$,
\begin{align*}
 \lvert j(2\tau)-\beta_0\rvert\lvert j(2\tau)-\beta_1\rvert\lvert j(2\tau)-\beta_2\rvert&\leq0.04\lvert\delta\rvert\lvert j\rvert^5\\
 \lvert j(2\tau)-\beta_0\rvert&\leq10^5\lvert\delta\rvert^{1/3}\\
 &\leq 1.2\cdot10^6\cdot2^{-P/3},
\end{align*}
and similarly for the nearest root of $g(z)$ to each of the roots of $f(z)$. As $\lvert\tilde{j}-1728\rvert\geq 2^{-P/3}$, $\lvert j-1728\rvert\geq 2^{-P/3}-1800\cdot2^{-P}\geq 2^{-P/3-1}$, and by Lemma \ref{j_to_tau_small}, $\lvert \tau - i\rvert\geq 2^{-P/6-9}$, so that by Lemma \ref{bounds_for_images}, 
\begin{equation*}
 \left\lvert j(2\tau) - j\left(\frac{\tau}{2}\right)\right\rvert\geq7.18\cdot10^7\cdot2^{- P/6-9}.
\end{equation*}
Now letting $\beta_i$ and $\beta_j$ be the nearest roots to $j(2\tau)$ and $j\left(\frac{\tau}{2}\right)$ respectively,
\begin{align*}
 \lvert\beta_i-\beta_j\rvert &= \left\lvert (\beta_i-j(2\tau))-j(2\tau)-\left(\left(\beta_j-j\left(\frac{\tau}{2}\right)\right)-j\left(\frac{\tau}{2}\right)\right)\right\rvert\\
 &\geq \left\lvert j(2\tau) - j\left(\frac{\tau}{2}\right)\right\rvert - \lvert j(2\tau)-\beta_i\rvert - \left\lvert j\left(\frac{\tau}{2}\right)-\beta_j\right\rvert\\
 &\geq 7.18\cdot10^7\cdot 2^{-P/6-9} - 1.2\cdot10^6\cdot2^{-P/3}\\
 &\geq 7.18\cdot10^7\cdot 2^{-P/6-9} - 0.01\cdot2^{-P/6}\\
 &\geq 1.4\cdot10^5\cdot2^{-P/6}.
\end{align*}
In particular, $\beta_i$ and $\beta_j$ are distinct, and given that, by Lemma \ref{bounds_for_images}, $\left\lvert j(2\tau)-j\left(\frac{\tau+1}{2}\right)\right\rvert\geq2.4\cdot10^5\geq7.18\cdot10^7\cdot2^{-P/6-9}$, one may similarly deduce a separation of $2\cdot10^{5}$ between the closest root to $j\left(\frac{\tau+1}{2}\right)$ and either of the other two. So each root $j_i$ of $\Phi_2(j,z)$ has a unique associated closest root $\beta_i$ of $\Phi_2(\tilde{j},z)$, which satisfies
\begin{equation*}
 \lvert j_i - \beta_i\rvert\leq 1.2\cdot10^6\cdot2^{-P/3}.
\end{equation*}
The relative precision of $\beta_i$ to its closest root is then bounded by 
\begin{equation*}
 \frac{1.2\cdot10^6\cdot2^{-P/3}}{\lvert j(2\tau)\rvert}, \frac{1.2\cdot10^6\cdot2^{-P/3}}{\lvert j\left(\frac{\tau}{2}\right)\rvert}\leq 2^{-P/3 + 3}
\end{equation*}
if $\beta_i$ is the root close to $j(2\tau)$ or $j\left(\frac{\tau}{2}\right)$, and 
\begin{equation*}
 \frac{1.2\cdot10^{6}\cdot2^{-P/3}}{\lvert j(\frac{\tau+1}{2})\rvert}\leq 2^{-P/3+10}
\end{equation*}
for the root close to $ j\left(\frac{\tau+1}{2}\right)$.
\end{proof}
Finally we show that Newton's method may be used to obtain an approximation to $j(2\tau)$.
\begin{proposition}
Let $j=j(\tau)$, where $\lvert\tau-i\rvert\leq2^{-30}$, and $\tilde{j}$ be an approximation to $j$ of relative precision $2^{-P}$, $P\geq 300$, such that $\lvert\tilde{j}-1728\rvert\geq2^{-P/3}$. Let
\begin{equation*}
 \tau_0 = i+\sqrt{\frac{\tilde{j}-1728}{j^{(2)}(i)}},
\end{equation*}
where the sign of the square root is chosen arbitrarily.
Then Newton's method applied to $\Phi_2(\tilde{j},z)$, with starting point $j(2\tau_0)$ will converge to either the root of $\Phi_2(\tilde{j},z)$ closest to $j(2\tau)$, or the root of $\Phi_2(\tilde{j},z)$ closest to $j\left(\frac{\tau}{2}\right)$, and after $[2\log P]$ steps will produce an approximation either $j(2\tau)$ or $j\left(\frac{\tau}{2}\right)$ of relative precision $2^{-P/3+11}$.
\end{proposition}
\begin{proof}
 Firstly, we bound the difference between $\tau$ and $\tau_0$. We let $\tau = i+\delta$. By Lemma \ref{taylor_series_at_i},
 \begin{equation*}
  \left\lvert j(\tau)-1728-\frac{j^{(2)}(i)}{2}\delta^2\right\rvert\leq0.07\lvert\delta\rvert^2,
 \end{equation*}
so that
\begin{equation*}
 \left\lvert\sqrt{\frac{2(\tilde{j}-1728)}{j^{(2)}(i)}}-\delta\right\rvert\left\lvert\sqrt{\frac{2(\tilde{j}-1728)}{j^{(2)}(i)}}+\delta\right\rvert\leq \frac{388}{j^{(2)}(i)}\lvert\delta\rvert^2+\frac{2^{-P+1}\lvert j\rvert}{j^{(2)}(i)}\leq 3\cdot10^{-6}\lvert\delta\rvert^2.
\end{equation*}
Now we will show that, taking some branch of the square root above, we will obtain a good starting point for Newton's method. Let $\epsilon = \sqrt{\frac{2(\tilde{j}-1728)}{j^{(2)}(i)}}$, where the branch of the square root is arbitrary. Firstly, one of $\left\lvert\epsilon-\delta\right\rvert$ or $\left\lvert\epsilon+\delta\right\rvert$ is $\geq\lvert\delta\rvert$, and so the other is $\leq 3\cdot10^{-6}\lvert\delta\rvert$. In the first case, $\lvert\tau_0-\tau\rvert\leq3\cdot10^{-6}\lvert\delta\rvert$, and in the second case,
\begin{equation*}
 \left\lvert\tau_0-\left(-\frac{1}{\tau}\right)\right\rvert=\left\lvert\frac{-1+i\delta+i\epsilon+1+\delta\epsilon}{\tau}\right\rvert\leq\lvert\delta+\epsilon\rvert+\lvert\delta\epsilon\rvert\leq3.1\cdot10^{-6}\lvert\delta\rvert.
\end{equation*}
Now as $\lvert\delta\rvert\leq2^{-30}$, $j'(z)$ is bounded in absolute value by $1.81\cdot10^6$ between $2\tau_0$ and either of $2\tau$ or $-\frac{2}{\tau}$, and the distance from $2\tau_0$ to the closest of these is bounded by $6.2\cdot10^{-6}\lvert\delta\rvert$, so either
\begin{equation*}
 \left\lvert j(2\tau_0)-j(2\tau)\right\rvert\leq 12\lvert\delta\rvert
\end{equation*}
or
\begin{equation*}
 \left\lvert j(2\tau_0)-j\left(-\frac{2}{\tau}\right)\right\rvert\leq 12\lvert\delta\rvert.
\end{equation*}
Now we bound the terms occurring in Kantorovich's criterion. Let $\beta_0,\beta_1,\beta_2$ be the roots of $\Phi_2(\tilde{j},z)$, with $\beta_0$ the closest root to $j(2\tau_0)$, and $\beta_1$ the other root near $j(2i)$. Firstly, as $\lvert j\rvert\geq2^{-P/2}$, by Lemma \ref{bounds_for_images}, the above, and Lemma \ref{phi2discrepancy_points_1728}, we have, for $2\tau_0$,  
\begin{align*}
  \left\lvert j(2\tau_0)-\beta_0\right\rvert&\leq 12\lvert\delta\rvert + 2^{-P/3+10}\leq12.1\lvert\delta\rvert\\
 \left\lvert j(2\tau_0)-\beta_1\right\rvert&\leq 7.25\cdot10^6\lvert\delta\rvert+2^{-P/3+10}\leq 7.26\cdot10^6\lvert\delta\rvert\\
 \left\lvert j(2\tau_0)-\beta_2\right\rvert&\leq3.1\cdot10^5+2^{-P/3+10}\leq3.11\cdot10^5\\
 \left\lvert j(2\tau_0)-\beta_1\right\rvert&\geq7.18\cdot10^6\lvert\delta\rvert-2^{-P/3+10}\geq 7.17\cdot10^6\lvert\delta\rvert\\
 \left\lvert j(2\tau_0)-\beta_2\right\rvert&\geq2.4\cdot10^5-2^{-P/3+10}\geq2.39\cdot10^5,
 \end{align*}
 and similarly, for any $\tau'$ satisfying $\lvert\tau'-2\tau\rvert\leq35\lvert\delta\rvert$ (which we take to ensure the condition on $r$ is satisfied), we have, as $\lvert j'(z)\rvert\leq1.81\cdot10^{6}$ between $2\tau_0$ and $\tau'$, and as $\lvert\delta\rvert\leq2^{-30}$,
 \begin{align*}
  \left\lvert j(\tau')-\beta_0\right\rvert&\leq 6.4\cdot10^{7}\lvert\delta\rvert+12\lvert\delta\rvert\leq0.06,\\
 \left\lvert j(\tau')-\beta_1\right\rvert&\leq 6.4\cdot10^{7}\lvert\delta\rvert+7.26\cdot10^6\lvert\delta\rvert\leq0.06,\\
 \left\lvert j(\tau')-\beta_2\right\rvert&\leq6.4\cdot10^{7}\lvert\delta\rvert+3.11\cdot10^5\leq3.2\cdot10^5.
 \end{align*}
For our bound on $\Phi_2(\tilde{j},z)$ evaluated at $j(\tau')$, we have
\begin{align*}
 \left\lvert\Phi_2(\tilde{j},j(2\tau_0))\right\rvert &= \lvert j(\tau')-\beta_0\rvert\lvert j(\tau')-\beta_1\rvert\lvert j(\tau')-\beta_2\rvert\\
 &\leq2.8\cdot10^{13}\lvert\delta\rvert^2.
\end{align*}
For the first derivative, we have
\begin{align*}
 \left\lvert\Phi_2'(\tilde{j},j(2\tau_0))\right\rvert &= ( j(\tau')-\beta_0)( j(\tau')-\beta_1)+( j(\tau')-\beta_0)( j(\tau')-\beta_2)+( j(\tau')-\beta_1)( j(\tau')-\beta_2)\\
 &\geq \lvert j(\tau')-\beta_1\rvert\lvert j(\tau')-\beta_2\rvert-\lvert j(\tau')-\beta_0\rvert\lvert j(\tau')-\beta_1\rvert-\lvert j(\tau')-\beta_0\rvert\lvert j(\tau')-\beta_2\rvert\\
 &\geq 7.17\cdot10^6\cdot2.39\cdot10^5\lvert\delta\rvert - 12.1\cdot7.26\cdot10^6\lvert\delta\rvert^2- 12\cdot3.11\cdot10^5\lvert\delta\rvert\\
 &\geq 1.7\cdot10^{12}\lvert\delta\rvert,
\end{align*}
and for the second derivative,
\begin{align*}
 \left\lvert\Phi_2''(\tilde{j},j(\tau'))\right\rvert &\leq 2\lvert j(\tau')-\beta_0\rvert+2\lvert j(\tau')-\beta_1\rvert+2\lvert j(\tau')-\beta_2\rvert\\
 &\leq 6.5\cdot10^5.
\end{align*}
These now give
\begin{equation*}
 \frac{\lvert\Phi_2(\tilde{j},j(2\tau_0))\rvert\lvert\Phi_2''(\tilde{j},j(\tau'))\rvert}{\lvert\Phi_2'(\tilde{j},j(2\tau_0))\rvert^2}\leq\frac{2.8\cdot10^{13}\lvert\delta\rvert^2\cdot6.5\cdot10^5}{(1.7\cdot10^{12}\lvert\delta\rvert)^2}\leq 2^{-17}<\frac{1}{2},
\end{equation*}
and for the condition on $r$, we have
\begin{align*}
 r&=35\lvert\delta\rvert,\\
 2\eta&=2\frac{\lvert\Phi_2(\tilde{j},j(2\tau_0))\rvert}{\lvert\Phi_2'(\tilde{j},j(2\tau_0))\rvert}\leq2\frac{2.8\cdot10^{13}\lvert\delta\rvert^2}{1.7\cdot10^{12}\lvert\delta\rvert}\leq34\lvert\delta\rvert,
\end{align*}
which ensures convergence. 
For the rate of convergence, we need a lower bound on the second derivative and an upper bound on the first derivative, which we have as follows,
\begin{align*}
 \left\lvert\Phi_2''(\tilde{j},j(\tau'))\right\rvert &\geq 2\lvert j(\tau')-\beta_2\rvert-2\lvert j(\tau')-\beta_0\rvert-2\lvert j(\tau')-\beta_1\rvert\\
 &\geq 4.7\cdot10^5,
\end{align*}
and
\begin{align*}
 \left\lvert\Phi_2'(\tilde{j},j(2\tau_0))\right\rvert &\leq \lvert j(2\tau_0)-\beta_1\rvert\lvert j(2\tau_0)-\beta_2\rvert+\lvert j(2\tau_0)-\beta_0\rvert\lvert j(2\tau_0)-\beta_1\rvert\\
 &\quad\quad\quad+\lvert j(2\tau_0)-\beta_0\rvert\lvert j(2\tau_0)-\beta_2\rvert\\
 &\leq 7.25\cdot10^6\cdot3.11\cdot10^5\lvert\delta\rvert + 12.1\cdot7.25\cdot10^6\lvert\delta\rvert^2+ 12.1\cdot3.11\cdot10^5\lvert\delta\rvert\\
 &\leq 2.3\cdot10^{12}\lvert\delta\rvert,
\end{align*}
which gives a bound on the convergence, for $k\geq 1$, of
\begin{equation*}
 \frac{1}{2^k}2^{-17\cdot2^k}\frac{\lvert\Phi_2'(\tilde{j},j(2\tau_0))\rvert}{\lvert\Phi_2''(\tilde{j},j(\tau'))\rvert}\leq2^{-17\cdot2^k}.
\end{equation*}
So in order to obtain an absolute precision of $2^{-P}$, $[2\log P]$ steps will suffice. By Lemma \ref{phi2discrepancy_points_1728}, this approximation to $\beta_0$ will then be an approximation to either $j(2\tau)$ or $j\left(-\frac{2}{\tau}\right)$ of relative precision $2^{-P/3+11}$.
\end{proof}

\subsubsection{$j$ near $0$}

We now bound the discrepancy between the roots of $\Phi_2(j,z)$ and $\Phi_2(\tilde{j},z)$ when $\left\lvert\tau-\frac{1+i\sqrt{3}}{2}\right\rvert\leq2^{-30}$ and $\lvert \tilde{j}\rvert\geq 2^{-P/2}$.
\begin{lemma}\label{phi2discrepancy_points_0}
 Let $j=j(\tau)$, where $\left\lvert\tau-\frac{1+i\sqrt{3}}{2}\right\rvert\leq2^{-30}$, and suppose that $\tilde{j}$ is an approximation to $j$ of absolute precision $2^{-P}$, with $P\geq 300$, and $\left\lvert \tilde{j}\right\rvert\geq 2^{-P/2}$. Then relative precision of any root of $\Phi_2(\tilde{j},z)$ to its closest root of $\Phi_2(j,z)$ is at most $2^{-P/3 + 2}$, and the roots of $\Phi_2(\tilde{j},z)$ are separated by at least $2^{-P/6+8}$.
\end{lemma}
\begin{proof}
Firstly, as in the proof of the Lemma \ref{phi2discrepancy_points_1728}, with $f(z) = \Phi_2(j,z)$ and $g(z) = \Phi_2(\tilde{j},z)$, we will bound the size of the coefficients of $f-g$. Let $\tilde{j}=j+\delta$. For the coefficient of $z^2$, we have
\begin{equation*}
 \lvert 2j\delta+\delta^2+1448\delta\rvert\leq 1500\lvert\delta\rvert,
\end{equation*}
for the coefficient of $z$,
\begin{equation*}
 \lvert 2976\delta j+1488\delta^2+40773375\delta\rvert\\
  \leq 4.1\cdot10^7\lvert\delta\rvert,
\end{equation*}
and for the constant term
\begin{equation*}
  \left\lvert 3\delta j^2+3\delta^2j+\delta^3-16200\delta j-16200\delta^2+8748000000\delta\right\rvert\leq 8.8\cdot10^9\lvert\delta\rvert.
 \end{equation*}
 Now evaluating $g(z)$ at $j(\tau')$, for $\tau'\in\{2\tau,\frac{\tau}{2},\frac{\tau+1}{2}\}$, as $\lvert j(\tau')\rvert\leq60000$, we have
 \begin{align*}
  \lvert g(j(2\tau))\rvert&\leq60000^2\cdot1500\lvert\delta\rvert + 60000\cdot4.1\cdot10^7\lvert\delta\rvert + 60000\cdot8.8\cdot10^9\lvert\delta\rvert\\
  &\leq 6\cdot10^{14}\lvert\delta\rvert.
 \end{align*}
Letting $\beta_0,\beta_1,\beta_2$ be the roots of $g(z)$, we have, as $\lvert\tilde{j}\rvert\geq2^{-P/2}$, which impels $\lvert j\rvert\geq 2^{-P/2-1}$,
\begin{align*}
 \lvert\beta_0-j(\tau_i)\rvert\lvert\beta_1-j(\tau_i)\rvert\lvert\beta_2-j(\tau_i)\rvert&\leq6\cdot10^{14}\lvert\delta\rvert\\
 &\leq2^{-P+50}\\
\end{align*}
Now for any $\beta_i$, letting $j(\tau_j)$ be the nearest root of $f(z)$, we have
\begin{equation*}
 \lvert\beta_i-j(\tau_j)\rvert\leq2^{-P/3+17}.
\end{equation*}
By Lemma \ref{j_to_tau_small}, as $\lvert j\rvert\geq2^{-P/2}$,  
\begin{equation*}
 \left\lvert \tau - \frac{1+i\sqrt{3}}{2}\right\rvert\geq 2^{-P/6-9 },
\end{equation*}
which yields, by Lemma \ref{bounds_for_images}, for $i\neq j$,
\begin{equation*}
 \lvert j(\tau_i)-j(\tau_j)\rvert\geq 2^{-P/6+9}.
\end{equation*}
Similarly to the previous lemma, we now observe that, with $\beta_i$ the closest root to $j(\tau_i)$,
\begin{align*}
 \lvert\beta_i-\beta_j\rvert &\geq \lvert j(\tau_i) - j(\tau_j)\rvert - 2\cdot2^{-P/3+17}\\
 &\geq 2^{-P/6+9} - 2^{-P/3+17}\\
 &\geq 2^{-P/6+8}
\end{align*}
and so $\beta_i$, $\beta_j$ are distinct. So each root has a unique closest associated root, with relative precision
\begin{equation*}
 \frac{2^{-P/3+17}}{\lvert j(\tau_j)\rvert} \leq 2^{-P/3+2}.
\end{equation*}
\end{proof}

\begin{proposition}
Let $j=j(\tau)$, where $\left\lvert\tau-\frac{1+i\sqrt{3}}{2}\right\rvert\leq2^{-31}$, and $\tilde{j}$ be an approximation to $j$ of relative precision $2^{-P}$ such that $\lvert\tilde{j}\rvert\geq2^{-P/2}$. Let
\begin{equation*}
 \tau_0 = \frac{1+i\sqrt{3}}{2}+\sqrt[3]{\frac{6\tilde{j}}{j^{(3)}(i)}},
\end{equation*}
where the branch of the cube root is chosen arbitrarily.
Then Newton's method applied to $\Phi_2(\tilde{j},z)$, with starting point $j(2\tau_0)$ will converge to either the root of $\Phi_2(\tilde{j},z)$ closest to $j(2\tau)$, the root of $\Phi_2(\tilde{j},z)$ closest to $j\left(\frac{\tau}{2}\right)$, or the root of $\Phi_2(\tilde{j},z)$ closest to $j\left(\frac{\tau+1}{2}\right)$, and after $[2\log P]$ steps will produce an approximation to either $j(2\tau)$, $j\left(\frac{\tau}{2}\right)$, or $j\left(\frac{\tau+1}{2}\right)$ of relative precision $2^{-P/3+3}$.
\end{proposition}
\begin{proof}
 Let $\tau = \frac{1+i\sqrt{3}}{2}+\delta$. By Lemma \ref{taylor_series_at_i},
 \begin{equation*}
  \left\lvert j(\tau) - \frac{j^{(3)}\left(\frac{1+i\sqrt{3}}{2}\right)}{6}\delta^3\right\rvert\leq 0.07\lvert\delta\rvert^3,
 \end{equation*}
so that
\begin{equation*}
  \left\lvert \frac{6\tilde{j}}{j^{(3)}\left(\frac{1+i\sqrt{3}}{2}\right)} -\delta^3\right\rvert\leq 1.6\cdot10^{-6}\lvert\delta\rvert^3.
\end{equation*}
We let $\epsilon = \sqrt[3]{\frac{6\tilde{j}}{j^{(3)}\left(\frac{1+i\sqrt{3}}{2}\right)}}$. Considering the geometry of the cube roots of $\delta^3$, the product of the two furthest from $\epsilon$ is at least $\lvert\delta\rvert^2$, so letting the closest be $\delta_0$, we have
\begin{equation*}
\lvert\epsilon - \delta_0\rvert\leq 1.6\cdot10^{-6}\lvert\delta\rvert.
\end{equation*}
We now show that any branch of the cube root taken for $\epsilon$ will result $2\tau_0-1$ being very close to one of the $\text{\normalfont SL}_2(\mathbb{Z})$--equivalent elements of $\mathcal{F}$ to one of $2\tau,\frac{\tau}{2}, \frac{\tau+1}{2}$. We have
\begin{align*}
 \left\lvert\frac{1+i\sqrt{3}}{2}+\epsilon- \left(-\frac{1}{\tau}+1\right)\right\rvert&=\frac{\left\lvert \tau+i\sqrt{3}\tau+2+2\epsilon\tau-2\tau\right\rvert}{2\lvert\tau\rvert}\\
 &=\frac{\left\lvert -\frac{1}{2}-i\frac{\sqrt{3}}{2}-\delta+i\frac{\sqrt{3}}{2} - \frac{3}{2}+i\sqrt{3}\delta+2+(1+i\sqrt{3})\epsilon+2\epsilon\delta\right\rvert}{2\lvert\tau\rvert}\\
 &=\frac{\left\lvert(1+i\sqrt{3})\epsilon-(1-i\sqrt{3})\delta+2\epsilon\delta\right\rvert}{\left\lvert2\tau\right\rvert}\\
 &\leq\lvert\epsilon - e^{4i\pi/3}\delta\rvert+\frac{\lvert\delta\epsilon\rvert}{2},
\end{align*}
and 
\begin{align*}
 \left\lvert\frac{1+i\sqrt{3}}{2}+\epsilon- \left(-\frac{1}{\tau-1}\right)\right\rvert&=\frac{\left\lvert \tau-1+i\sqrt{3}(\tau-1)+2+2\epsilon(\tau-1)\right\rvert}{2\lvert\tau\rvert}\\
 &=\frac{\left\lvert -\frac{1}{2}+i\frac{\sqrt{3}}{2}+(1+i\sqrt{3})\delta-i\frac{\sqrt{3}}{2}-\frac{3}{2}+2+(-1+i\sqrt{3})\epsilon +2\epsilon\delta\right\rvert}{2\lvert\tau\rvert}\\
 &=\frac{\left\lvert(1-i\sqrt{3})\epsilon-(1+i\sqrt{3})\delta-2\epsilon\delta\right\rvert}{\left\lvert2\tau\right\rvert}\\
 &\leq\lvert\epsilon - e^{2i\pi/3}\delta\rvert+\frac{\lvert\delta\epsilon\rvert}{2}.
\end{align*}
In particular, the difference between $2\tau_0-1$ and the closest of the $\text{\normalfont SL}_2(\mathbb{Z})$--equivalent elements of $\mathcal{F}$ to $2\tau,\frac{\tau}{2}, \frac{\tau+1}{2}$, which we let be $\tau_1$, and the others $\tau_2,\tau_3$, is bounded by
\begin{equation*}
 2\left\lvert\epsilon-\delta_0\right\rvert+\lvert\delta\epsilon\rvert\leq3.3\cdot10^{-6}\lvert\delta\rvert.
\end{equation*}
Now as the absolute value of the derivative of $j(z)$ is bounded in absolute value by $3.4\cdot10^5$ between $\tau_0$ and $\tau_1$, we have
\begin{equation*}
 \left\lvert j(2\tau_0)-j(\tau_1)\right\rvert\leq 1.2\lvert\delta\rvert.
\end{equation*}
Now by the bound of Lemma \ref{phi2discrepancy_points_0} on the discrepancy of the roots of $\Phi_2(\tilde{j},z)$ and $\Phi_2(j,z)$, and the bounds on the distances between the distinct roots of Lemma \ref{bounds_for_images}, letting $\beta_0$ be the closest root of $\Phi_2(\tilde{j},z)$ to $j(2\tau_0)$, and $\beta_1,\beta_2$ the other two roots, we collect the relevant bounds for Kantorovich's criterion,
\begin{align*}
 \left\lvert j(2\tau_0) - \beta_0\right\rvert&\leq 1.2\lvert\delta\rvert\\
 \left\lvert j(2\tau_0) - \beta_1\right\rvert, \left\lvert j(2\tau_0) - \beta_2\right\rvert&\leq 1.35\cdot10^6\lvert\delta\rvert\\
 \left\lvert j(2\tau_0) - \beta_1\right\rvert, \left\lvert j(2\tau_0) - \beta_2\right\rvert&\geq 1.14\cdot10^6\lvert\delta\rvert,
\end{align*}
and with $r=4\lvert\delta\rvert$, for any $\tau'$ such that $\lvert\tau'-2\tau_0\rvert\leq4\lvert\delta\rvert$, we have, as the derivative of $j$ is bounded in absolute value by $3.4\cdot10^{5}$ here,
\begin{align*}
 \left\lvert j(\tau') - \beta_0\right\rvert&\leq 1.37\cdot10^6\lvert\delta\rvert\\
 \left\lvert j(\tau') - \beta_1\right\rvert, \left\lvert j(\tau') - \beta_2\right\rvert&\leq 2.71\cdot10^6\lvert\delta\rvert\\
\end{align*}
Now we bound the various terms of Kantorovich's criterion. Firstly,
\begin{align*}
 \lvert\Phi_2(\tilde{j},j(2\tau_0))\rvert&\leq 1.2\lvert\delta\rvert\cdot(1.35\cdot10^6\lvert\delta\rvert)^2\\
 &\leq 2.2\cdot10^{12}\lvert\delta\rvert^3,
\end{align*}
for the first derivative, we have
\begin{align*}
 \lvert\Phi_2'(\tilde{j},j(2\tau_0))\rvert&\geq (1.14\cdot10^6\lvert\delta\rvert)^2 - 2\cdot1.2\lvert\delta\rvert\cdot1.35\cdot10^6\lvert\delta\rvert\\
 &\geq 1.2\cdot10^{12}\lvert\delta\rvert^2,
\end{align*}
and for the second derivative,
\begin{align*}
 \lvert\Phi_2''(\tilde{j},j(\tau'))\rvert&\leq 2\cdot1.37\cdot10^6\lvert\delta\rvert+4\cdot2.71\cdot10^6\lvert\delta\rvert\\
 &\leq1.36\cdot10^7\lvert\delta\rvert,
\end{align*}
so that we have
\begin{equation*}
 \frac{\lvert\Phi_2(\tilde{j},j(2\tau_0))\rvert\lvert\Phi_2''(\tilde{j},j(2\tau_0))\rvert}{\lvert\Phi_2'(\tilde{j},j(2\tau_0))\rvert^2}\leq\frac{2.2\cdot10^{12}\lvert\delta\rvert^3\cdot1.36\cdot10^7\lvert\delta\rvert}{(1.2\cdot10^{12}\lvert\delta\rvert^2)^2}\leq 2^{-15}<\frac{1}{2},
\end{equation*}
and for the condition on $r$ we have
\begin{align*}
 r &=4\lvert\delta\rvert\\
 2\eta&=2\frac{\lvert\Phi_2(\tilde{j},j(2\tau_0))\rvert}{\lvert\Phi_2'(\tilde{j},j(2\tau_0))\rvert}\leq3.8\lvert\delta\rvert,
\end{align*}
ensuring convergence. For the rate of convergence, we have the upper bound on the first derivative as follows,
\begin{align*}
 \left\lvert\Phi_2'(\tilde{j},j(2\tau_0))\right\rvert \leq 1.4\cdot10^{12}\lvert\delta\rvert^2.
\end{align*}
For the lower bound on the second derivative, we have that
\begin{equation*}
 \lvert\Phi_2''(\tilde{j},j(\tau'))\rvert = \lvert6 j(\tau') - 2(\beta_0+\beta_1+\beta_2)\rvert.
\end{equation*}
Now the second term in the absolute value is the coefficient of $z^2$ of $\Phi_2(\tilde{j},z)$, which is equal to 
\begin{equation*}
 -\tilde{j}^2+1488\tilde{j}-162000.
\end{equation*}
By Lemma \ref{taylor_series_at_i}, $
\lvert j\rvert\leq\left(\frac{\left\lvert j\left(\frac{1+i\sqrt{3}}{2}\right)\right\rvert}{6}+0.07\right)\lvert\delta\rvert^3$, and as $\lvert\tilde{j}-j\rvert\leq2^{-P}\leq\lvert j\rvert^2$, $\lvert\tilde{j}\rvert\leq1.01\lvert j\rvert$, so
\begin{equation*}
 \lvert-\tilde{j}^2+1488\tilde{j}\rvert\leq4.13\cdot10^8\lvert\delta\rvert^3\leq0.1\lvert\delta\rvert.
\end{equation*}
In particular, $-(\beta_1+\beta_2+\beta_3)=162000+\theta_1$, where $\lvert\theta_1\rvert\leq 0.1\lvert\delta\rvert$.
As $\lvert\delta\rvert\leq2^{-31}$, $\lvert\tau'-2\tau_0\rvert\leq4\lvert\delta\rvert$, and $\lvert2\tau_0-1-i\sqrt{3}\rvert\leq(2+3.3\cdot10^{-6})\lvert\delta\rvert$, 
\begin{equation*}
 \lvert\tau'-1-i\sqrt{3}\rvert\leq 6.1\lvert\delta\rvert\leq2^{-28},
\end{equation*}
and hence by Lemma \ref{taylor_series_at_2i},
\begin{equation*}
 \lvert j(\tau') - j(i\sqrt{3})\rvert\geq \lvert  j'(i\sqrt{3})\delta\rvert-1.3\lvert\delta\rvert\geq 334000\lvert\delta\rvert.
\end{equation*}
In particular, as $j(i\sqrt{3})=54000$,  $j(2\tau_0) = 54000+\theta_2$, where $\lvert\theta_2\rvert\geq334000\lvert\delta\rvert$.
Now combining these bounds, we have
\begin{align*}
  \lvert 6 j(\tau') - 2(\beta_0+\beta_1+\beta_2)\rvert &=\lvert 324000+6\theta_2-324000+2\theta_1\rvert\\
  &\geq 6\lvert\theta_2\rvert - 2 \lvert\theta_1\rvert\\
  &\geq6\cdot334000\lvert\delta\rvert - 0.2\lvert\delta\rvert\\
  &\geq2\cdot10^6\lvert\delta\rvert.
\end{align*}
Returning to the convergence, we have
\begin{equation*}
 \frac{\lvert\Phi_2'(\tilde{j},j(2\tau_0))\rvert}{\lvert\Phi_2''(\tilde{j},j(\tau'))\rvert}\leq\frac{1.4\cdot10^{12}\lvert\delta\rvert^2}{2\cdot10^6\lvert\delta\rvert}\leq2^{-10}, 
\end{equation*}
so a bound for the rate of convergence is
\begin{equation*}
 2^{-15\cdot2^k}.
\end{equation*}
In particular, to obtain an absolute precision of $2^{-P}$, $[2\log P]$ steps will suffice.  The approximant obtained will then, by Lemma \ref{phi2discrepancy_points_0}, be an approximation of one of $j(2\tau),j\left(\frac{\tau}{2}\right)$, or $j\left(\frac{\tau+1}{2}\right)$, of relative precision at least $2^{-P/3+3}$.
\end{proof}

\subsubsection{Running time}
Now we have shown the time required to obtain an approximation to one of $j(2\tau),j\left(\frac{\tau}{2}\right),$ or $j\left(\frac{\tau+1}{2}\right)$ is $O(M(P)\log P)$, and the obtained value may then be used as an input to Newton's method on the compact set described in the subsequent section. In order to determine which of $j(2\tau),j\left(\frac{\tau}{2}\right),j\left(\frac{\tau+1}{2}\right)$ was computed above, after obtaining an inverse $\tau_0$, we compute $j(2\tau_0)$ and $j\left(\frac{\tau_0}{2}\right)$, and return the argument corresponding to whichever was closest to our initial approximation $\tilde{j}$ to $j$, after applying elements of $\text{\normalfont SL}_2(\mathbb{Z})$ to move it to the fundamental domain (which takes $O(P)$ time). In particular we will obtain an approximation of precision at least $2^{-P/3+12}$.

\subsection{Newton iteration on the compact set}

For $\tau$ such that $\tau \in \mathcal{F}$, $\text{\normalfont Im}(\tau)\leq 3.1$, $\lvert \tau - i\rvert\geq2^{-32}$, $\left\lvert\tau-\frac{1+i\sqrt{3}}{2}\right\rvert\geq 2^{-32}$, we make use of an algorithm due to Dupont, \cite{fastj} for the quasilinear evaluation of $j(\tau)$ to relative precision, in order to invert $j(z)$ by the secant method. As we are considering a compact set, there is some fixed precision our starting points for the secant method may be in order to obtain convergence for any $\tau$.
We compute the low precision inverses of $j$ by the formula
\begin{align*}
 \tau_0 &= i\frac{_2F_1\left(\frac{1}{6},\frac{5}{6},1,\frac{1}{2}+\frac{1}{2}\sqrt{1-\frac{1728}{\tilde{j}}}\right)}{_2F_1\left(\frac{1}{6},\frac{5}{6},1,\frac{1}{2}-\frac{1}{2}\sqrt{1-\frac{1728}{\tilde{j}}}\right)},\\
 \tau_1 &= \tau_0 - \frac{j(\tau_0)-j}{j'(\tau_0)},
\end{align*}
and we note that as $j$ is bounded away from $1728$, the arguments of the Gaussian hypergeometric functions are bounded away from their singularities at $0$ and $1$, and $j'$ is bounded away from $0$, so computation of the two points to a fixed precision takes a fixed amount of time. Upon a failure of convergence (or slow convergence), we may simply increase the precision of our starting points and repeat the process -- as there is a uniform bound on the required precision, this requires only constant time. As the secant method converges quadratically,  and the evaluation of $j(z)$ via the algorithm of $\cite{fastj}$ takes $O(M(P)\log P)$ time, the computational complexity of obtaining an approximation of precision $2^{-P+1}$ requires time $O(M(P)(\log P)^2)$. 

A similar analysis to that of Lemma \ref{taylor_series_at_i} gives a bound of \begin{equation*}
\left\lvert j^{k}(z_0)\right\rvert\leq3\cdot10^8\cdot14^kk!                                                                                  \end{equation*}
for $z_0$ in our compact set, so that if $z_1=z_0+\lvert\delta\rvert$, with $\lvert\delta\rvert\leq2^{-150}$, 
\begin{align*}
 \lvert j(z_0)+j'(z_0)\delta-j(z_1)\rvert&\leq3\cdot10^8\sum_{n=2}^\infty\lvert\delta\rvert^n14^n\\
 &\leq3\cdot10^8\cdot14^2\cdot2^{-150}\cdot\lvert\delta\rvert\sum_{n=0}^\infty2^{-150n}14^n\\
 &\leq5\cdot10^{-35}\lvert\delta\rvert.
\end{align*}
It may be verified that $\lvert j'(z)\rvert\geq10^{-19}$ in our compact set, so
\begin{equation*}
 \lvert j(z_0)-j(z_1)\rvert\geq \lvert j'(z_0)\rvert\lvert\delta\rvert-5\cdot10^{-35}\lvert\delta\rvert\geq 10^{-20}\lvert\delta\rvert,
\end{equation*}
and in particular once we have computed by the secant method $\tau^*$ such that $\lvert \tilde{j}-j(\tau^*)\rvert\leq2^{-P}$, $\lvert j-j(\tau^*)\rvert\leq3\cdot10^8\cdot2^{-P}$, and so, with $j(\tau)=j$, and $\tau^*=\tau+\delta$, we have 
\begin{equation*}
 \lvert\delta\rvert\leq 3\cdot10^8\cdot10^{20}\cdot2^{-P}\leq2^{-P+100}.
\end{equation*}
\subsection{$j$ very close to $0$ or $1728$}
Now if $\lvert\tilde{j}\rvert\leq 2^{-P/2}$, or $\lvert\tilde{j}-1728\rvert\leq2^{-P/3}$, we simply return $\frac{1+i\sqrt{3}}{2}$ or $i$ respectively, and this will, by Lemma \ref{j_to_tau_small}, be an approximation to the inverse of $j$ of absolute, and as $\lvert\tau\rvert\geq 1$, relative precision $2^{-P/6}$.

\section{Testing CM}

 Firstly, given the input of the $j$--invariant of an elliptic curve $E$, its degree $d$, and a bound on its height $H\geq e^e$, we may bound the maximum discriminant from which $j$ may arise, if $E$ were to have complex multiplication. We first consider $\lvert D\rvert\geq16$. The principal form of a discriminant has an associated $\tau$ of imaginary part $\frac{i \sqrt{\lvert D\rvert}}{2}$, so as $\lvert D\rvert\geq 16$, $M(j(\tau))\geq e^{\pi\sqrt{\lvert D\rvert}}-2079 \geq e^{3.13\sqrt{\lvert D\rvert}}$. In particular, as $M(j)\leq H^d$, 
 \begin{equation*}
  \lvert D\rvert\leq \frac{d^2(\log H)^2}{9.7}.
 \end{equation*}
 Now the $\tau$ associated to a binary quadratic form of discriminant of absolute value $\leq N$ has real part bounded in height by $2N$, and the square of the imaginary part bounded in height by $4N^2$, so we will determine if the preimage of $j$ is a quadratic irrational satisfying these conditions.

 We first bound the degree of  $j(z)$ at a quadratic irrational of discriminant $D$. For a fundamental discriminant $D$, we have the bound (Proposition 2.2, \cite{paulin}),
 \begin{equation*}
  h(D)\leq\frac{1}{\pi}\sqrt{\lvert D\rvert}(2+\log\lvert D\rvert),
 \end{equation*}
 and in the case of non--fundamental discriminants, by Theorem 7.4 of \cite{buell}, for an odd prime $p$,
 \begin{equation*}
  h(p^2D)\leq (p+1)h(D)\leq\frac{p+1}{p\pi}\sqrt{\lvert p^2D\rvert}(2+\log\lvert D\rvert)\leq\frac{4}{3\pi}\sqrt{\lvert p^2D\rvert}(2+\log\lvert p^2D\rvert)
 \end{equation*}
 and for $2$, if $D\not\equiv0(8)$,
 \begin{equation*}
 h(4D)\leq 3h(D)\leq \frac{3}{2\pi}\sqrt{\lvert 4D\rvert}(2+\log\lvert D\rvert)\leq \frac{3}{2\pi}\sqrt{\lvert 4D\rvert}(2+\log\lvert 4D\rvert)
 \end{equation*}
and if $D\equiv 0(4)$,
\begin{equation*}
 h(4D)\leq 2h(D)\leq\frac{1}{\pi}\sqrt{\lvert 4D\rvert}(2+\log\lvert D\rvert)\leq\frac{1}{\pi}\sqrt{\lvert 4D\rvert}(2+\log\lvert 4D\rvert).
\end{equation*}
In particular, we have the bound
\begin{equation*}
 h(D)\leq\frac{3}{2\pi}\sqrt{\lvert D\rvert}(2+\log\lvert D\rvert).
\end{equation*}
 For the Mahler measure of $j(\tau)$, as $\lvert j(\tau)\rvert\leq e^{2\pi\text{\normalfont Im}(\tau)}+2079\leq9.1e^{2\pi\text{\normalfont Im}(\tau)}$, we have the upper bound in terms of the reduced binary quadratic forms $ax^2+bxy+cy^2$ of discriminant $D$,
 \begin{equation*}
  M(j(\tau))\leq9.1^{h(D)}\exp\left(2\pi\sqrt{\lvert D\rvert}{\sum_{\substack{(a,b,c)\\ \text{ reduced}}}\frac{1}{a}}\right).
 \end{equation*}
 The number of times each $a$ may occur is bounded by the number of solutions of $b^2\equiv d(2a)$, which is bounded by twice the number of distinct divisors $r(a)$ of $a$. By \cite{divisor_problem}, we have the bound
 \begin{equation*}
  A(x) = 2\sum_{1\leq a\leq x}r(A)\leq 2x\log x+0.4x+2x^{1/2},
 \end{equation*}
 so that by Abel's summation formula,
 \begin{align*}
  \sum_{1\leq a\leq h(D)}\frac{r(a)}{a} &= \frac{A(h(D))}{h(D)}+\int_{1}^{h(D)}\frac{A(y)}{y^2}dy\\
  &\leq \frac{(\log h(D))^2}{2}+1.2\log(h(D))+2.2\\
  &\leq 1.3\log\lvert D\rvert^2,
 \end{align*}
 yielding a bound of
 \begin{equation*}
  M(j(\tau))\leq9.1^{h(D)}e^{5.2\pi\sqrt{\lvert D\rvert}(\log\lvert D\rvert)^2}\leq e^{5.9\pi\sqrt{\lvert D\rvert}(\log\lvert D\rvert)^2}.
 \end{equation*}
 In particular, by Liouville's inequality, if $j\neq j(\tau)$, where $\tau$ is of discriminant $D$,
 \begin{align*}
  \lvert j - j(\tau)\rvert &\geq 2^{-\frac{3d}{2\pi}\sqrt{\lvert D\rvert}(2+\log\lvert D\rvert)} H^{-\frac{3d}{2\pi}\sqrt{\lvert D\rvert}(2+\log\lvert D\rvert)}e^{5.9\pi d\sqrt{\lvert D\rvert}(\log\lvert D\rvert)^2},
 \end{align*}
 and substituting our bound on $\lvert D\rvert$, we obtain the lower bound
\begin{equation*}
 \exp\left(-30d^2\log H(\log d+\log\log H)^2\right).
\end{equation*}
  Now as $\lvert j\rvert\leq H^d$, letting $j=j(z_0)$, where $z_0\in\mathcal{F}$, as $\lvert j(z_0)\rvert\geq e^{2\pi\text{\normalfont Im}(z_0)}-2079$, we have $\text{\normalfont Im}(z_0)\leq \frac{d\log H}{2\pi}+2$. By differentiating the $q$--expansion of $j$, and taking an upper bound, (note that the sum of the absolute values of the terms of the tail is decreasing), we have, as $d\log H\geq0$,
  \begin{equation*}
   \lVert j'\rVert_{B(z_0,1)}\leq 2\pi e^{d\log H+6\pi}+800\leq e^{d\log H+21},
  \end{equation*}
so that if $\lvert z_0 - \tau\rvert\leq1$, 
\begin{equation*}
 \lvert j - j(\tau)\rvert\leq\lvert z_0 - \tau\rvert e^{d\log H+21}.
\end{equation*}
Combining this with our separation result, noting that $\log d+\log\log H\geq1$, if
\begin{equation*}
 \lvert z_0 - \tau\rvert\leq \exp\left(-31d^2\log H(\log d+\log\log H)^2-21\right),
\end{equation*}
then $j=j(\tau)$, and so $j$ is a singular modulus. Now as our method of inverting $j$ from an input of regulated precision $2^{-P}$ obtains its inverse with relative precision at least $2^{-P/6}$, to obtain a result of absolute precision $2^{-Q}$, as $\lvert \tau\rvert\leq\frac{d\log H}{2\pi}+2$, we will need an input of relative precision $2^{-6Q-2\log(d\log H+2)}$. In particular, with an input of relative precision
\begin{equation*}
 2^{-300d^2\log H(\log d+\log\log H)^2-200},
\end{equation*}
our computed approximation $\tilde{z_0}$ to $z_0$ will have sufficient precision to determine whether
\begin{equation*}
 \lvert z_0 - \tau\rvert\leq\lvert z_0 - \tilde{z_0}\rvert+\lvert\tau-\tilde{z_0}\rvert
\end{equation*}
is sufficiently small.
In order to obtain our candidate $\tau$, we develop the continued fractions of the real part and the square of the imaginary part of our computed inverse $\tilde{\tau}$ -- as the discriminants under consideration are bounded by $\lvert D\rvert$, the distance between the real parts of any two preimages of the singular moduli is at least 
\begin{equation*}
 19^{-1}d^{-4}(\log H)^{-4},
\end{equation*}
and of the squares imaginary parts is at least
\begin{equation*}
 19^{-2}d^{-8}(\log H)^{-8}.
\end{equation*}
So we develop the continued fractions of the real part of $\tilde{z_0}$ and the square of the imaginary part of $\tilde{z_0}$ until their difference from the convergents is bounded by  
\begin{equation*}
 19^{-3}d^{-8}(\log H)^{-8},
\end{equation*}
for once this holds of the convergents, they will form the only possible quadratic irrational inverse of $j$.
If the height of the real convergent $c_r$ is greater than $\frac{d^2(\log H)^2}{9.7}$, or the height of the square of the imaginary convergent $c_i$ is greater than $\frac{d^4(\log H)^4}{90}$, or the square root of $c_i$ is a rational number, then we may conclude that $j$ is not a singular modulus -- otherwise, letting $\tau = c_r+i\sqrt{c_i}$, we may test if $\lvert \tilde{z_0}-\tau\rvert$ satisfies our condition for $j=j(\tau)$. If so, then $j$ is a singular modulus, and otherwise $j$ is not a singular modulus for $\tau$ of discriminant $\lvert D\rvert\geq16$. It is clear that the computational complexity of the calculation of the convergents and other operations in this algorithm is dominated by that of inverting $j$, so we obtain a running time, with $T=d^2\log H(\log d+\log\log H)^2$ of 
\begin{equation*}
 O(M(T)(\log T)^2).
\end{equation*}
For testing discriminants with $\lvert D\rvert\leq16$, we may simply take the list of all $\tau$ of discriminant $\leq16$, and test whether $j$ is equal to them. The degree of $j(\tau)$ for such $\tau$ is bounded by $2$, and its Mahler measure is bounded by $3\cdot10^6$, so if 
\begin{equation*}
 \lvert j-j(\tau)\rvert\leq2^{-2d}H^{-2d}\left(3\cdot10^6\right)^{-d}\leq\exp(2d(33+\log H)),
\end{equation*}
then $j=j(\tau)$. So with an approximation $\tilde{j(\tau)}$ to $j(\tau)$ of absolute precision $2^{-4d(33+\log H)+2}$, then we will be able to establish whether $j=j(\tau)$. By the algorithm to compute $j$ of \cite{fastj}, which requires time $O(M(P)\log P)$ for relative precision of $P$, as $\lvert j(\tau)\rvert\leq 3\cdot10^{6}$, computing $j$ to the required absolute precision possible in time
\begin{equation*}
 O(M(d\log H)(\log d+\log\log H)),
\end{equation*}
which again is $O\left(M(T)(\log T)^2\right)$.

\end{document}